\documentclass[12pt]{article}
\usepackage{hyperref}
\usepackage{graphicx}
\usepackage{amssymb,amsmath,amsthm}
\usepackage{bbold}
\usepackage{bm}
\usepackage{pgfplots}
\usepackage{tkz-graph}
\usepackage{fixmath}
\usepackage{relsize}
\def\epsilon{\varepsilon}
\def\bp{\pi}

\def\Rf{{\cal R}_f}
\def\br^2{{\Bbb R}^2}
\def\gr{}

\def\cT{{\cal T}}

\usepackage{accents}
\newcommand{\doubletilde}[1]{\accentset{\approx}{#1}}

\GraphInit[vstyle = Shade]
\renewcommand*{\VertexBallColor}{green!60}
\renewcommand*{\VertexTextColor}{black}

\tikzset{
 LabelStyle/.style = { rectangle, 
            rounded corners, 
            draw, 
            minimum width = 10em, 
            fill = yellow!50, 
            text = black, },
 VertexStyle/.style = { inner sep=3pt, 
             minimum size = 5pt,
             shape     = \VertexShape,
             ball color   = \VertexBallColor,
             color     = \VertexLineColor,
             inner sep   = \VertexInnerSep,
             outer sep   = 0.5\pgflinewidth,
             text      = \VertexTextColor,
             minimum size  = \VertexSmallMinSize,
             line width   = \VertexLineWidth},
 EdgeStyle/.append style = {->, 
               double=yellow, 
               color=orange}
}

\usepackage{enumitem}

\newcommand{\inter}[1]{%
 {\kern0pt#1}^{\mathrm{o}}%
}

\usepackage{color}
\newcommand{\blue}{\textcolor{blue}}
\newcommand{\red}{\textcolor{red}}

\newcommand{\violet}{\textcolor{black}}
\newcommand{\magenta}{\textcolor{black}}
\newcommand{\teal}{\textcolor{black}}
\newcommand{\olive}{\textcolor{black}}

\usepackage{enumitem}

\definecolor{jim}{rgb}{1,.4,0}
\definecolor{jim2}{rgb}{0,.6,0}

\newcommand{\allred}{\color{red}{}}
\newcommand{\allblack}{\color{black}{}}
\newcommand{\allviolet}{\color{black}{}}

\newcommand{\allteal}{\color{black}{}}
\newcommand{\allmagenta}{\color{black}{}}

\newcommand{\eg}{{\it{e.g.}}}

\def\J{J}
\def\Jall{\J_{all}}
\def\Jall{J^{int}}

\usepackage{accents}
\newlength{\dhatheight}

\newcommand{\beq}{\begin{linenomath}\begin{equation}} 
\newcommand{\eeq}{\end{equation}\end{linenomath}} 

\def\bydef{{\buildrel \rm def \over =}}

\def\phi{\varphi}

\def\cK{{\cal K}}

\def\cR{{\cal R}}
\def\cS{{\cal S}}

\newtheorem{definition}{Definition}
\newtheorem{notation}{Notation}
\newtheorem{theorem}{Theorem}
\newtheorem*{thmA}{Theorem A}

\newtheorem{plc}{Plc}[section]

\newtheorem{theoremX}{Theorem}

\newtheorem{lemma}[plc]{Lemma}
\newtheorem{corollary}[plc]{Corollary}
\newtheorem{example}{Example}[section]
\newtheorem{proposition}[plc]{Proposition}

\title{
 Graph and backward asymptotics\\ of the tent map
}
\author{Ana Anu\v{s}i\'{c}\\ \small Nipissing University, North Bay, ON (Canada)\\ \\ Roberto De Leo\\ \small Howard University, Washington DC 20059 (USA)}
\pgfplotsset{compat=1.16} 
\begin{document}
\maketitle

\begin{abstract}
    The tent \teal{map family} is arguably the simplest \teal{1-parametric} family of maps with non-trivial dynamics and it is still an active subject of research. 
    In recent works
    the second author, jointly with J. Yorke, studied the graph and backward limits of S-unimodal maps. In this article we generalize those results to tent-like unimodal maps. By tent-like here we mean maps that share fundamental properties that characterize tent maps, namely \teal{unimodal maps without wandering intervals nor attracting cycles and whose graph has a finite number of nodes.} 
\end{abstract}
\section{Introduction}
%

The qualitative dynamics of continuous and discrete dynamical systems can be encoded into a graph. 
The original idea goes back to S. Smale. 
In~\cite{Sma67}, he proved 
that the non-wandering set $\Omega_f$ of a Axiom-A diffeomorphism $f$ is the union of disjoint, closed, invariant  indecomposable 
subsets $\Omega_i$ on each of which $f$ is transitive.
Based on this decomposition, he associated to each such $f$ a directed graph whose nodes are the $\Omega_i$ and there is an edge from $\Omega_i$ to $\Omega_j$ if and only if $W^s(\Omega_i)\cap W^u(\Omega_j)\neq\emptyset$, where $W^s$ and $W^u$ denote, respectively, the stable and unstable manifolds.

A shortcoming of Smale's construction above is that it does not extends well to more general settings. 
In particular, when the non-wandering set is not hyperbolic, the $\Omega_i$ are not necessarily disjoint and therefore the graph is not well-defined. 
This happens, for instance, in case of the logistic map at the endpoints of any window of the bifurcation diagram.
Perhaps this is the reason why, while a Spectral Decomposition Theorem for unimodal maps was established long ago in several versions by several authors~\cite{JR80,vS81,HW84,BL91,Blo95}, the graph component of Smale's construction was not pursued.

A natural fix for this problem is replacing the non-wandering set by the chain-recurrent set $\cR_f$ introduced by C.~Conley in~\cite{Con78}. 
The set $\cR_f$ has two strong advantages over $\Omega_f$. 
First, $\cR_f$ has a built-in natural equivalence relation that decomposes it naturally into equivalence classes $N_i$ (that we call nodes) that are closed and invariant under $f$.
This is the analogue of the Smale's decomposition of $\Omega_f$ in the $\Omega_i$, with the advantage that the $N_i$ are always pairwise disjoint. 
Moreover, chain-recurrence is the widest possible definition of recurrence, as proved first by Conley himself in case of continuous compact dynamical systems and later by several other authors in more general cases, including discrete~\cite{Nor95b} and infinite-dimensional systems~\cite{Ryb87}: either a point is chain-recurrent, so that its orbit is contained within one of the nodes $N_i$, or its orbit asymptotes forward to some $N_i$ and, if it has backward orbits, each of its backward orbits asymptotes to some node $N_j$ with $j\neq i$. 

In~\cite{DLY20}, J. Yorke and the second author associated to any given continuous map $f$ a directed graph $\Gamma_f$ whose nodes are the equivalence classes $N_i$ of its chain-recurrent set and such that there is an edge from node $N_i$ to node $N_j$ if, arbitrarily close to $N_i$, there are points that asymptote forward to $N_j$ under $f$.
Notice that $\Gamma_f$ is always acyclic.
In the same article, 
the authors proved that the graph of a S-unimodal map is a tower, namely that all $p+1$ nodes $N_i$ of its graph (possibly countably infinitely many) can be sorted in a linear order so that, for each $j>i$, there are points arbitrarily close to node $N_i$ that asymptote to $N_j$. 
Notice that, in this notation, $N_0$ is the fixed boundary point and $N_p$ is the attractor. 
\allblack

In ~\cite{DL22}, the second author used the results of~\cite{DLY20} to study the $s\alpha$-limits in S-unimodal maps. The most relevant result for our goals is the following theorem, that we state below for maps whose attractor is not of type $A_5$ (see Prop.~7 in~\cite{DLY20} and Prop.~\ref{prop:A}). In case of the logistic map family, for instance, this means that the map is not at the right endpoint of any window in the bifurcation diagram. 
\begin{theoremX}[$s\alpha$-limits of S-unimodal maps~\cite{DL22}]\label{thm:A}
    Let $f$ be a S-unimodal map with $p+1$ nodes \violet{$N_i$}, $1\leq p<\infty$, and whose attractor $N_p$ is not of type $A_5$.
    Then there are $p+1$ closed sets $V_i$ such that:
    \begin{enumerate}
    \item $V_0=[0,f(c)]$;
    \item $V_p=N_p$;
    \item if $x\in V_p$, then $s\alpha(x)=\cup_{k=0}^p N_k$;
    \item for $i<p$,
      \begin{itemize}
        \item $V_{i+1}\subset V_{i}$;
        \item $N_i\subset V_i\setminus V_{i+1}$;
        \item if $x\in V_i\setminus V_{i+1}$, then $s\alpha(x)=\cup_{k=0}^i N_k$.
      \end{itemize}
    \end{enumerate}
\end{theoremX}
In short, this means that points close enough to $N_i$ have $s\alpha$-limit equal to the union of all nodes from $N_0$ to $N_i$. 

\medskip
\olive{
In the present article we extend the results above to a wide class of less regular unimodal maps.
We drop smoothness and only require the following topological properties: continuity; absence of wandering intervals and attracting cycles; finitely many nodes.
We call these maps {\em T-unimodal}.
}

\olive{
The model we have in mind is the tent map $T_s$, \teal{whose graph in $[0,1/2]$ is a segment joining $(0,0)$ with $(1/2,s/2)$ and is symmetric with respect to $x=1/2$ (see Figure~\ref{fig:tent})}.
This map 
is a natural piecewise-linear version of the logistic map and it has been widely studied since the seminal work by Milnor and Thurston~\cite{MT88}.
Since it is not smooth, results from~\cite{DLY20,DL22} do not apply to it.
}

\begin{figure}[!ht]
 	\centering
 	\begin{tikzpicture}[scale=8]
 	\draw[thin] (0,0)--(0,1)--(1,1)--(1,0)--(0,0);
 	\draw[thick] (0,0)--(0.5,1.8/2)--(1,0);
 	\node[below] at (0,0) {\small $0$};
 	\node[below] at (0.5,0) {\small $c=\frac{1}{2}$};  
 	\draw[dashed] (0.5,0)--(0.5,0.9);
 	\node[below] at (1,0) {\small $1$};
 	
 	\node[left] at (0,0) {\small $0$};
 	\node[left] at (0,1.8/2) {\small $\frac{s}{2}$};
 	\draw[dashed] (0.5,0.9)--(0,0.9);
 	\node[left] at (0,1) {\small $1$};
 	
 	\draw[very thin] (0.18,0.18)--(0.18,0.9)--(0.9,0.9)--(0.9,0.18)--(0.18,0.18);
 	\draw[dashed] (0,0)--(1,1);
 	\end{tikzpicture}
 	
 	\caption{The graph of $T_{1.8}$ and its core.}
 	\label{fig:tent}
 \end{figure}
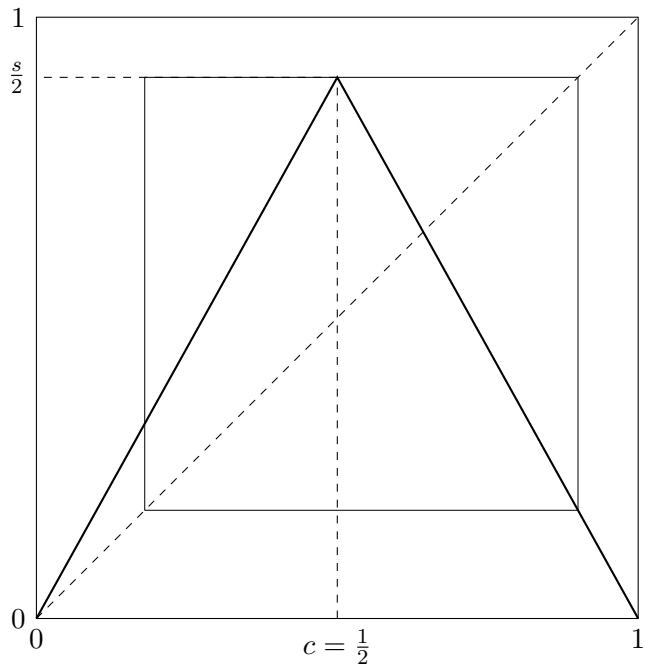

\medskip \noindent
The main results of the present article are the following:
\begin{enumerate}
    \item The graph of a T-unimodal map is a tower (Thm~\ref{thm:tower});
    \item Each point in the attractor $A$ of a T-unimodal map has a backward orbit dense in $A$ (Thm~\ref{thm:dense});
    \item Theorem A holds for T-unimodal maps (Thm.~\ref{thm:main}).
\end{enumerate}
In particular, in case of the tent map, we get the following result.
%
    For any $s$ such that  $\log_2 s\in[2^{-p},2^{1-p})$, $p\geq1$, the graph of $T_s$ is a tower of $p+1$ nodes 
    \olive{$N_0, N_1,\dots, N_p$, where: 
    \begin{enumerate}
    \item $N_0$ is the fixed endpoint; 
    \item $N_k$, $1\leq k<p$, is a repelling $2^{k-1}$-cycle; 
    \item the attractor $N_p$ is the core of $N_{p-1}$ (see Def.~\ref{def:Ncore}), namely a cycle of $2^{{p-1}}$ intervals. 
    \end{enumerate}}
    

%


\section{Definitions and basic results}
\label{sec:dr}
A discrete dynamical system on a metric space $(X, d)$ is given by the iterations of a continuous map $f : X \to X$. In this article, we are interested in
the case where $X$ is a closed interval and $f$ is a unimodal map, as defined below.
\begin{definition}
    A $C^0$ map $f:[a,b]\to[a,b]$ is {\em unimodal} if:
    \begin{enumerate}
        \item $f(a)=f(b)=a$ or $f(a)=f(b)=b$;
        \item there is a point $c\in(a,b)$ such that $f$ is strictly increasing (or decreasing) in $[a,c]$ and strictly decreasing (or, respectively, increasing) in $[c,b]$. 
    \end{enumerate}
\end{definition}
In all statements and examples throughout the
article, we will assume that $c$ is a maximum. Of
course, the same proofs hold also when $c$ is a minimum after trivial modifications that we leave to the reader.

One of the simplest examples of unimodal map families is the \violet{family of tent maps}
$$
T_s(x) = s\frac{1-|1-2x|}{2},
$$
with $s\in[0,2]$ and $x\in[0,1]$.
\olive{A way of thinking about it is as  a piece-wise linear version of the logistic map family $\ell_\mu(x)=\mu x(1-x)$, a family of central importance in unimodal dynamics because every S-unimodal map is conjugated to a logistic one. 
The tent map family is the main motivation of this article, since it is not granted a priori that the results for S-unimodal maps mentioned in the introduction extend to this family too.}

\begin{figure}
 \centering
 \includegraphics[width=14cm]{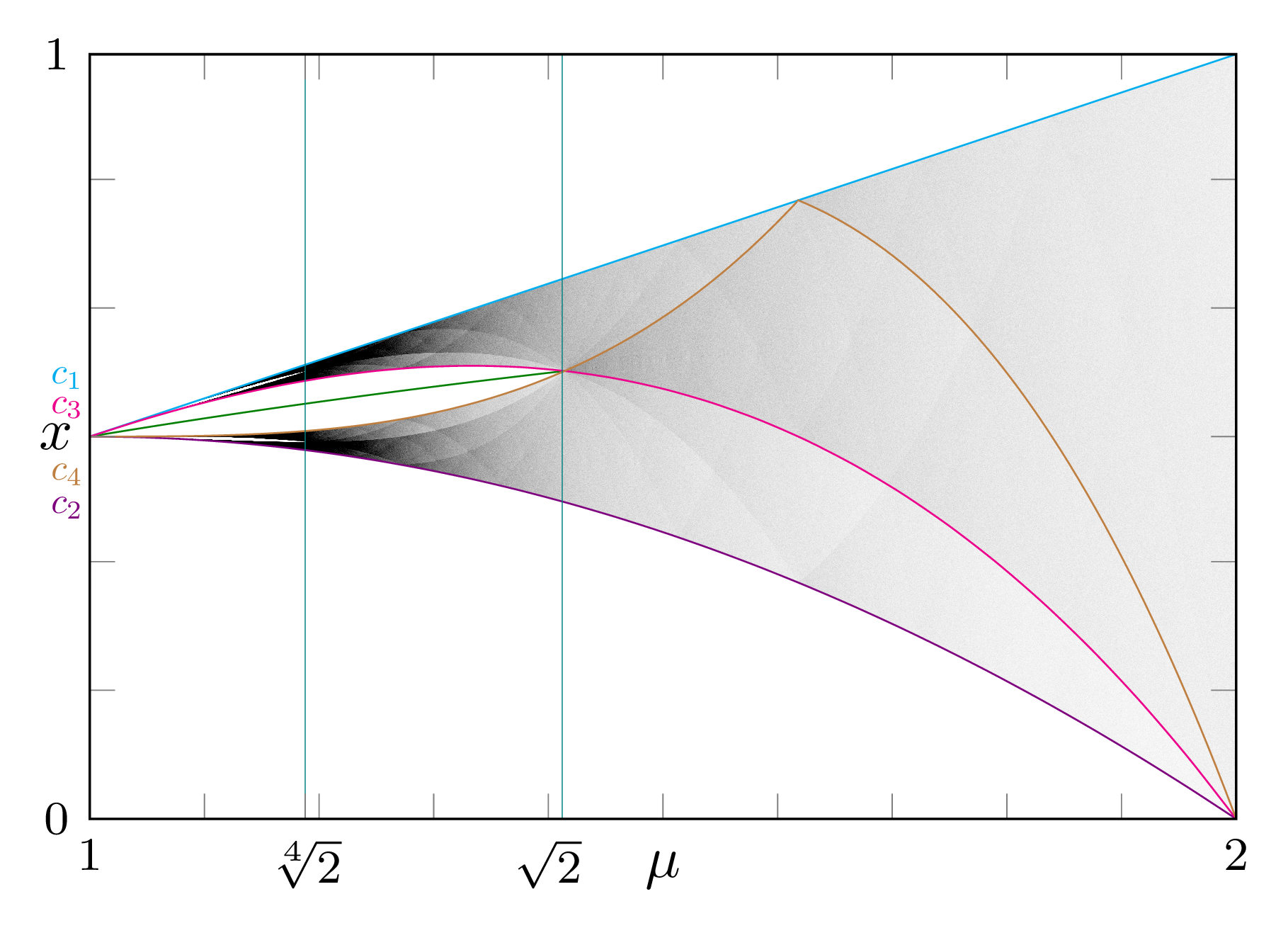}
 \caption{{\em Bifurcation diagram of the tent map.} 
 This picture shows the bifurcation diagram of the tent map $T_s$ in the parameter range $s\in(1,2]$.
 Attractors are painted in shades of gray (depending on the density) and repelling periodic orbits in green. The colored lines labeled by $c_k$ are the lines $T^k_s(c)$, where $c = 0.5$ is the critical point of $T_s$. 
 For $s\in[\sqrt{2},2)$, the attractor is the core of $T_s$, namely the tight trapping region $[c_2,c_1]$, and there is no repelling node besides $N_0=\{0\}$ (for $s=2$, the fixed endpoint merges into the attractor and there is a single node); for $s\in[\sqrt[4]{2},\sqrt{2})$, there is a second repelling node $N_1=\{\bp_s=\frac{s}{s+1}\}$, the internal fixed point of $T_s$, and the attractor is the core of $N_1$, namely the tight trapping region $[c_3,c_1]\cup[c_2,c_4]$.
 At non-rigorous level, one can think of this bifurcation diagram as what we get from the one of the logistic map (e.g. see~\cite{DLY20}) by first collapsing to the single point $x=c$ the Cantor set attractor at the Feigenbaum-Myrberg point and then collapsing to a single point in parameter space all windows. The ``ghosts'' of the windows in the tent map diagram are the parameter values for which the critical point is periodic.
}
 \label{fig:tm}
\end{figure} 
\medskip
Notice that, given any unimodal map $f$, the equation $f(x) = r$ has two distinct solutions for each $r<f(c)$. This justifies the following:
\begin{notation}
    For each $p\neq c$ in the domain of $f$, we denote by $\hat p$ the solution different from $p$ of the equation $f(x) = f(p)$. \teal{We say that $p$ and $\hat p$ are conjugated with respect to $f$.}
\end{notation}

\medskip\noindent
{\bf Chain-recurrence and nodes.} \allviolet
The idea of encoding the qualitative behaviour of a dynamical system into a graph goes back to Smale~\cite{Sma67}, in case of Morse functions on a compact manifold. 
In that case, and in general in case of {\em Axiom-A} diffeomorphisms $f$, the nodes of the graph are {\em closed, disjoint, invariant, and indecomposable} subsets of the set of non-wandering points of $f$, that we denote by $\Omega_f$. Recall that a point $x$ is non-wandering for $f$ if, for every neighbourhood $U$ of $x$, there is an integer $n>0$ such that $f^n(U)\cap U\neq\emptyset$. 

A decade later, Conley~\cite{Con78} introduced the more general concept of {\em chain-recurrence} and showed that it is more suitable than that of non-wandering point to extend the original idea of Smale to general discrete and continuous dynamical systems.
One of the main reasons for this is the existence of a natural equivalence relation among chain-recurrent points.
\allblack
\begin{definition}[Bowen, 1975~\cite{Bow75}]
  An {\em $\epsilon$-chain} from $x$ to $y$ is a sequence of points $x_0=x,x_1,\dots,x_n,x_{n+1}=y$
  such that $d(f(x_i),x_{i+1})<\epsilon$ for all $i=0,\dots,n$. We say that $x$ is {\bf downstream} (resp. {\bf upstream}
  from $y$) if, for every $\epsilon>0$, there is an $\epsilon$-chain from $y$ to $x$ (resp. from $x$ to $y$).
  A point $x\in X$ is {\bf chain-recurrent} if it is downstream from itself.
\end{definition}
We denote the set of all chain-recurrent points of $X$ under $f$ by $\Rf$ . The relation 
$$
\hbox{$x \sim y$ if and only if $x$ is
both upstream and downstream from $y$} 
$$
is an equivalence relation in $\Rf$ (e.g. see \cite{Nor95b}). 
%
Note that $\Omega_f\subset\Rf$ but, in general, the inverse does not hold, as shown by the example below. 
\allviolet
\begin{example}
  Let $\mu_0\simeq3.857$ be the parameter value of the right endpoint of the period-3 window of the logistic map.
  Set $c_k=\ell_{\mu_0}^k(c)$. 
  Then every point in $[c_2,c_1]$ is chain-recurrent (indeed, $\ell_{\mu_0}$ has only two nodes: $N_0=\{0\}$ and $N_1=[c_2,c_1]$) but the only non-wandering points within $[c_2,c_1]$ are the points of the attractor $A=[c_2,c_5]\cup[c_3,c_6]\cup[c_4,c_1]$ and the points of a repelling Cantor set (see Fig.~5 in~\cite{DL22}, where the Cantor set is painted in red). 
  Notice that, at $\mu=\mu_0$, the attractor is a {\em cyclic} trapping region, at the boundary of which lie the unstable 3-cycle $\{c_4,c_5,c_6\}$. The full intersection between the attractor and the repelling Cantor set is this 3-cycle plus the pre-periodic point $c_3=1-c_6$.
\end{example}
\begin{notation}
  For any $x\in X$, we denote by $\omega_f(x)$ the set of all accumulation points of the orbit of $x$ under $f$.
\end{notation}
\begin{proposition}[Norton, 1995~\cite{Nor95b}]
  Given any $x\in X$, $\omega_f(x)\subset\cR_f$ and $y\sim z$ for each $y,z\in\omega_f(x)$.
\end{proposition}
\begin{definition}[Conley, 1978~\cite{Con78}]
  We call graph $\Gamma_f$ of a dynamical system $f$ on $X$ the directed graph having as nodes the elements of $\cR_f/\sim$ and having an edge from node $N_1$ to node $N_2$ if, arbitrarily close to $N_1$, there are points $x$
  such that $\omega_f(x)\subset N_2$.
\end{definition}
Because of the definition above, from now on we will refer to chain-recurrent equivalence classes simply as {\bf nodes}.
Notice that, as a subset of $X$, each node is closed and invariant under $f$ \cite{Nor95b}.

\allviolet
\medskip\noindent
%
\begin{definition}[Milnor, 1985~\cite{Mil85}]
  A closed invariant set $A\subset[a,b]$ is an {\bf attractor} if it satisfies the following conditions:
  \begin{enumerate}
        \item the basin of attraction of $A$, namely the set of all $x\in[a,b]$ such that $\omega(x)\subset A$, has strictly positive measure;
        \item there is no strictly smaller invariant closed subset $A'\subset A$ whose basin differs from the basin of $A$ by just a zero-measure set.
    \end{enumerate}
\end{definition}
\begin{definition}
    We call a node an {\bf attracting node} if it contains an attractor, otherwise we call it a {\bf repelling node}. 
\end{definition}
%
\allblack

\medskip\noindent
{\bf T-unimodal maps.} 
\allviolet 
We highlight here the main dynamical properties that hold for the tent map family. Our main results are true for every map with such property. 

\begin{definition}
    Let $f$ be a unimodal map and $J\subset[a,b]$ a closed interval. We say that $J$ is a {\bf wandering interval} if:
   \begin{enumerate}
       \item the intervals $f^k(J)$, $k=0,1,\dots$, are mutually disjoint;
       \item $J$ is not in the basin of a periodic orbit.
   \end{enumerate}
   We say that $J$ is a {\bf homterval} for $f$ if $c$ is not in the interior of $f^k(J)$ for any integer $k\geq0$,
    namely if the restriction of $f$ to each $f^k(J)$ is monotonic.
\end{definition}
The two definitions above are related by the following well-known result (see Lemma 3.1 in Ch.~2 of~\cite{dMvS93}).
\begin{proposition}
  \label{prop:homtervals}
  Let $J$ be a homterval for $f$. Then either one of the following holds:
  \begin{enumerate}
  \item $J$ is a wandering interval;
  \item for each $x\in J$,   $\omega_f(x)$ is a periodic orbit. 
  \end{enumerate}
\end{proposition}
\begin{proof}
    \olive{
    Set $J_k=f^k(J)$.
    Notice that $f^{k'}(J_k)=J_{k+k'}$.
    If $J$ is not a wandering interval,
    then there are $n\geq0$ and $m>0$ such that $J_n\cap J_{n+m}\neq\emptyset$, so that
    $J_n\cup J_{n+m}$ is connected.
    Then, similarly, $J_{n+m}\cap J_{n+2m}\supset f^m(J_n\cap J_{n+m})\neq\emptyset$ and so 
    $J_n\cup J_{n+m}\cup J_{n+2m}$ is connected.
    Ultimately, this argument shows that the closure of $\cup_{k=0}^\infty J_{n+km}$, denoted by $H$, is an interval.
    This interval is forward invariant under $f^m$ by construction and does not contain $c$ since $J$ is a homterval, so $f^m|_H$ is a homeomorphism and the claim follows.
    } 
\end{proof}
\begin{corollary}
  A map without wandering intervals and attracting periodic orbits cannot have homtervals.
\end{corollary}
\begin{example}
  If $f$ is unimodal and $f(c)<c$, every interval $H\subset \violet{[a,b]}$ is trivially a homterval. In case of the tent map $T_s$, this happens when $s\in[0,1)$.
\end{example}
\begin{example}
 \olive{
  Consider the logistic map $\ell_\mu$ for $\mu\in(1,2)$. Then the internal fixed point $\pi_\mu=1-1/\mu$ is attracting and regular, namely $\ell'_\mu(\pi_\mu)>0$.
  In particular, $\pi_\mu<c=1/2$.
  The interval $H=[0,\pi_\mu]$ is forward invariant and does not contain $c$, so  is a homterval. 
  Every point of $H$ except $x=0$ converges to $\pi_\mu$.
  }
\end{example}%
\olive{
In fact, the homterval of the example above is the only type of homterval that can arise in S-unimodal maps (e.g. see the proof of Prop.~3 in~\cite{DLY20}).}

\medskip
Throughout this article, we focus on the subset of unimodal maps $f$ satisfying the following assumption.

\medskip\noindent
{\bf Assumption (T).} The map $f$ has:
\begin{enumerate}
    \item \violet{no wandering intervals};
    \item no attracting cycles;
    \item a finite number of nodes.
\end{enumerate}
\begin{definition}
    We say that a unimodal map is {\em T-unimodal} if it satisfies Assumption (T). 
\end{definition}
Below we show that, like in case of S-unimodal maps, also T-unimodal maps have a single attractor.
%
%
\begin{definition}
    \label{def:trap}
    Let $f$ be a unimodal map. 
    By a {\bf period-$k$ trapping region} for $f$ we mean a collection $\cT$ of $k$ intervals with disjoint interiors $J_1,\dots,J_k$ 
    such that:
    \begin{enumerate}
    \item $c\in J_1$;
    \item $f(J_i)\subset J_{i+1}$ for each $i=1,\dots,k$, where we set $J_{k+1}\bydef J_1$.
    \end{enumerate}
    We use the notation $J_i(\cT)$ to denote the $i$-th interval of the trapping region $\cT$ and $\Jall(\cT)$ for the union of the interiors of all $\J_i(\cT)$.
    We say that $\cT$ is {\bf tight} when $f^k(J_1(\cT))=J_1(\cT)$.
    We say that it is {\bf cyclic} when the endpoints of $J_1(\cT)$ are conjugated and one of them is periodic.
    In that case, we denote by $\gamma(\cT)$ the periodic orbit to which the periodic endpoint of $J_1(\cT)$ belongs to.
\end{definition}
\begin{example}
  Consider the tent map $T_s$. For $s\in(1,2]$,
  the interval 
  $J_1=[T^2_s(c),T_s(c)]$
  is a period-1 tight trapping region for $T_s$. 
  Indeed 
  $$
  T_s(c)=s/2>c,\;T^2_s(c)=(2-s)T_s(c)<c,\;T^3_s(c)=sT^2_s(c)>T^2_s(c),
  $$
  so that $c\in J_1$ and
  $$
  T_s(J_1) = T_s([T^2_s(c),c]\cup[c,T_s(c)])=[T^3_s(c),T_s(c)]\cup[T^2_s(c),T_s(c)]=J_1.
  $$
  Notice that this trapping region  
  is cyclic only for $s=2$, since $T_s(c)$ is fixed only for $s=1$ and $T^2_s(c)$ is fixed only for $s=1,2$. 
\end{example}
\begin{example}
    \label{ex:cyclic}
    \allteal
    Consider the tent map $T_s$.
    For $\sqrt{2}\leq s<2$, $T_2$ has a single repelling node, $N_0=\{0\}$. 
    In this case, there is only one cyclic trapping region: 
    $$\cT_0=\{[0,1]\}.$$ 
    For all $1<s<\sqrt{2}$, $T_s$ has also a second repelling node $N_1=\{\bp_s\}$, the internal fixed point, painted in green and labeled $p_1$ in Fig.~\ref{fig:tr}. 
    Correspondingly, $T_s$ has a second, period-2, trapping region 
    $$
    \cT_1=\{J_1=[q_1,p_1],J_2=[p_1,q_2]\},
    $$ 
    painted in red, where $q_1=\hat p_1$ and $q_2$ is the unique solution to $f^2(x)=p_1$ such that $x>p_1$. Since the only roots of $f(x)=p_1$ are $p_1$ itself and $q_1$, it follows that $f(q_2)=q_1$
    
    For all $1<s<\sqrt[4]{2}$, $T_s$ has also a third repelling node $N_2$ equal to the period-2 orbit whose points are labeled by $p'_1,p'_2$ in Fig.~\ref{fig:tr}.
    Notice that we labeled by $p'_1$ the closest to $c$ of the two points.
    Correspondingly, $T_s$ has a third, period-4, trapping region
    $$
    \cT_2=\{J'_1=[p'_1,q'_1],J'_2=[p'_2,q'_2],J'_3=[q'_3,p'_1],J'_4=[q'_4,p'_2]\},
    $$ where $q'_1=\widehat{p'_1}$, $q'_4=f(q'_3)$, $q'_3=f(q'_2)$ and $q'_2$ is the unique solution to $f^4(x)=p'_2$ such that $x>p'_2$. Note that it follows that $f(q'_4)=q'_1$.
\end{example}
%
\begin{figure}
 \centering
 \includegraphics[width=12cm]{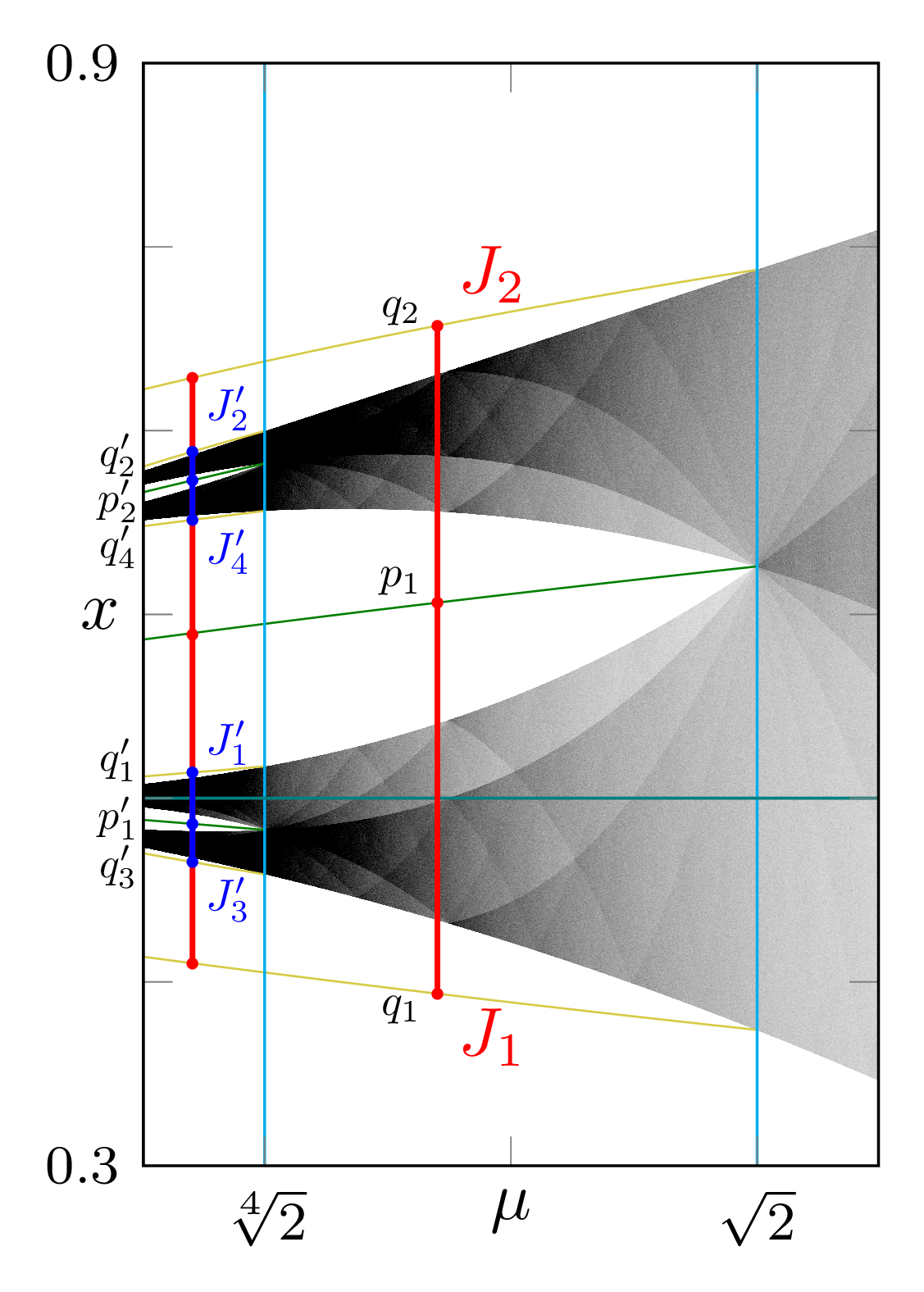}
 \caption{{\em Examples of cyclic trapping regions for the tent map.}
 This picture shows a part (not in scale) of the bifurcation diagram of the tent map family. See Ex.~\ref{ex:cyclic} for a thorough description of the picture.
}
 \label{fig:tr}
\end{figure} 

\allblack
%
%

\noindent
{\bf Attractor of a T-unimodal map.} The uniqueness of the attractor of a T-unimodal map follows at once from the fundamental result below.
\begin{thmA}[Jonker and Rand, 1980~\cite{JR80}; see also Thm~4.1 in~\cite{dMvS93}]\label{thm:JonkerRand}
  Let $f$ be a unimodal map with critical point $c$. Then its attractors can be only of the following three types:
  \begin{enumerate}
      \item a periodic orbit;
      \item a trapping region with a dense orbit; 
      \item a Cantor set on which $f$ acts as an adding machine. This attractor contains $c$ and the orbit of $c$ is dense in it. \teal{In this case, and only in this case, $f$ has infinitely many nodes}.
  \end{enumerate}
  \allteal
  If there is no attractor of type 2 or 3, then one of the attracting periodic orbits has $c$ in its basin of attraction.
\end{thmA}
\begin{corollary}
    \allteal
    Every unimodal map without attracting periodic orbits has a single attractor.
\end{corollary}
Throughout the article, we refer informally to attractors of type (2) as {\em chaotic attractors}, since the dynamics of $f$ on this type of attractor is chaotic according to both Li-Yorke and Devaney (e.g. see~\cite{AK01}).
\begin{corollary}\label{cor:Tunim}
    Every T-unimodal map $f$ has a single attractor and this single attractor is chaotic.
\end{corollary}

\allviolet

\noindent
{\bf Basic properties of the tent map family.}

\begin{definition}
    \label{def:core}
  We call {\bf core} of a unimodal map $f$ the interval $[f^2(c),f(c)]$.
\end{definition}
\begin{notation}
    We denote the attractor of a T-unimodal map $f$ by $A_f$, the core of $f$ by $K(f)$ and we use the notation $c_k$ for the point $f^k(c)$. In this notation, $K(f)=[c_2,c_1]$ (see Fig.~\ref{fig:tent}).
\end{notation}
%
%
\begin{proposition}
  \label{prop:core}
  The core of any T-unimodal map $f$ is a period-1 tight trapping region.
\end{proposition}
\begin{proof}\ 
    Notice first that we must have $\min\{c,c_3\}\geq c_2$.

    \allteal
    Assume indeed first that $c_2>c$. 
    Then $f$ would be injective and order-reversing on its core, so that the internal fixed point of $f$ would be inside $f([c_2,c_1])=[c_2,c_3]$ and so necessarily attracting.
    This is against the hypothesis that $f$ is T-unimodal.

    \allblack
    Assume now that $c_2>c_3$. By our previous argument, we have that $[c_3,c_2]$ lies at the left of $c$, so that $f$ is injective and order-preserving on it. Hence 
    $$
    f([c_3,c_2]) = [c_4,c_3]
    $$
    and so on. This means that the orbit of $c$ is converging to a fixed point, contradicting the fact that $f$ has no attracting periodic orbits. 
    
    The argument above shows that, in case of a T-unimodal, map we always have $c_1\geq \min\{c,c_3\}\geq c_2$. 
    Hence
    $$
    f([c_2,c_1])=f([c_2,c]\cup[c,c_1])=[c_3,c_1]\cup[c_2,c_1]=[c_2,c_1],
    $$
    namely $[c_2,c_1]$ is a period-1 tight trapping region.



\end{proof}
%
%
\allblack
The conditions that define T-unimodal maps are strong but satisfied by many non-trivial unimodal maps. One is the set of all S-unimodal maps with a chaotic attractor. A second one, on which this article focuses, is the tent map family. The theorem below shows that, for $s\in(1,2]$, $T_s$ satisfies Assumption (T).  

%

%
\begin{definition}
    Let $\Gamma$ be the graph of a dynamical system and $k$ either an integer or $\infty$.
    By {\bf $k$-cascade} of
    $\Gamma$ we mean a subgraph of $\Gamma$ 
    with nodes $N_1,\dots,N_k$ such that:
    \begin{enumerate}
    \item there are edges from each $N_i$ to all $N_j$ with $j>i$;
    \item $N_i$ is a periodic orbit of period $2^{i-1}$.
    \end{enumerate}
    %
\end{definition}
\begin{definition}
    A map $f:X\to X$ is {\bf topologically exact} if, for every open subset $U\subset X$, there is an integer $n\geq1$ such that $f^n(U)=X$.  
\end{definition}
Next theorem collects several basic results on the tent map family.
\begin{theoremX}\label{thm:B}
    The following properties hold for the tent map family:
    \begin{enumerate}
        \item If $s\in[0,1)$, $\Omega_{T_s}$ consists of a single fixed point: the attracting endpoint $x=0$.
        \item If $s=1$, $\Omega_{T_s}$ consists of an interval of fixed points, $\Omega_{T_s}=[0,\frac{1}{2}]$.
        \item If $\log_2 s\in[2^{-p},2^{1-p})$, the attractor $A$ is chaotic
        and $\Omega_{T_s}$ is the disjoint union of $A$ with
        the repelling endpoint $x=0$ and a $p-1$-cascade~\cite{JR80,vS88}. 
         \item $\Omega_{T_2}=[0,1]$~\cite{UvN47}.
        \item For $s>1$, $T_s$ has no homtervals~\cite{JR80}, no attracting cycles~\cite{JR80} and the restriction of 
        $T_s^{\,2^{p-1}}$ to any component of its attractor is topologically exact~\cite{JR80,BDOET91}. 
    \end{enumerate}
\end{theoremX}
\begin{proof} Note that $0$ is a fixed point of $T_s$ for all $s\in[0,2]$. If $s>1$, there is another fixed point $\bp=\frac{s}{s+1}$. We denote by $c=1/2$, and $c_n=T_s^n(c)$ for all $n\geq 1$.
    \begin{enumerate}
        \item If $s<1$, then $T_s(x)<x$ for all $x\in(0,1]$, so $0$ attracts all points in $(0,1]$. 
        \item If $s=1$, then $T_s(x)=x$ for all $x\in[0,c]$. If $x>c$, then $T_s^2(x)=\hat x<c$.
        \item If $1<s< 2$, then all $x\in(0,c_2)\cup (c_1,1]$ are attracted to the core $[c_2,c_1]$. Thus, $\Omega_{T_s}=\{0\}\cup\Omega_{T_s|_{[c_2,c_1]}}$. 
        
        If $s\in(\sqrt 2,2]$, then $T_s|_{[c_2,c_1]}\colon[c_2,c_1]\to[c_2,c_1]$ is topologically exact, see \eg \cite[Lemma~2]{BDOET91}. To see that, note that, if $J$ is an interval in $[c_2,c_1]$, then either $T^2_s(J)=[c_2,c_1]$ or the length of $T_s^2(J)$ is greater than the length of $J$. Thus, if $s\in(\sqrt 2,2]$, then $\Omega_{T_s|_{[c_2,c_1]}}=[c_2,c_1]$. 
        
        If $1<s\leq\sqrt 2$, then the core decomposes as $[c_2,c_1]=[c_2,\bp]\cup[\bp, c_1]=J_1\cup J_2$, where $T_s(J_1)=J_2$, $T_s(J_2)=J_1$, and $T_s^2|_{J_1}$, and $T_s^2|_{J_2}$ are topologically conjugate to $T_{s^2}|_{[0,c_1]}$. 
        From this the claim in point 3 follows inductively. 
        For $s=\sqrt 2$, $T_s^2|_{J_1}$ is conjugate to $T_2$, which implies $\Omega_{T_s|_{[c_2,c_1]}}=[c_2,c_1]$. For $s\in(\sqrt[4]{2},\sqrt 2)$, $T_s^2|_{J_1}$ is conjugate to $T_{s^2}|_{[0,c_1]}$, with the repelling fixed point $\bp$, and the core $[c_2,c_4]$, on which it is topologically exact. Thus $\Omega_{T_s|_{[c_2,c_1]}}=[c_2,c_4]\cup\{\bp\}\cup[c_3,c_1]$. See Figure~\ref{fig:tents}.
    
        
        
        \item Note that $T_2$ is topologically exact on $[0,1]$, so $\Omega_{T_2}=[0,1]$.
        \item Follows directly from points 3 and 4.
    \end{enumerate}
\end{proof}

 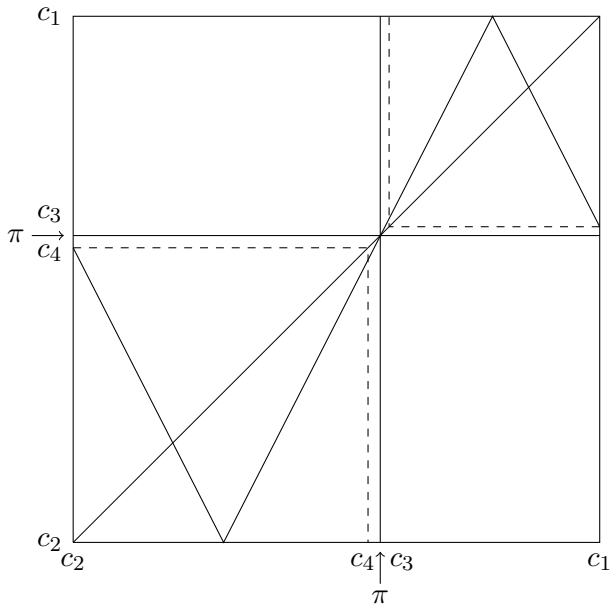
\begin{figure}[!ht]
 	\centering
 	\begin{tikzpicture}[scale=25]
 	\draw (0.42,0.42)--(0.42,0.7)--(0.7,0.7)--(0.7,0.42)--(0.42,0.42);
 	\draw[very thin] (0.42,0.42)--(0.7,0.7);
 	\draw (0.42,0.5768)--(0.5,0.42)--(0.643,0.7)--(0.7,0.588);
 	\draw (0.42,0.5833)--(0.5833,0.5833)--(0.5833,0.42);
 	\draw[dashed] (0.42,0.5768)--(0.5768,0.5768)--(0.5768,0.42);
 	\draw (0.7,0.5833)--(0.5833,0.5833)--(0.5833,0.7);
 	\draw[dashed] (0.7,0.588)--(0.588,0.588)--(0.588,0.7);
 	\node[below] at (0.42,0.42) {\small $c_2$};
 	\node[below] at (0.574,0.42) {\small $c_4$};
 	\draw[->] (0.5833,0.398)--(0.5833,0.415);
 	\node[below] at (0.5833,0.4) {\small $\bp$};
 	\node[below] at (0.595,0.42) {\small $c_3$};
 	\node[below] at (0.7,0.42) {\small $c_1$};
 	
 	\node[left] at (0.42,0.42) {\small $c_2$};
 	\node[left] at (0.42,0.574) {\small $c_4$};
 	\draw[->] (0.398,0.5833)--(0.415,0.5833);
 	\node[left] at (0.4,0.5833) {\small $\bp$};
 	\node[left] at (0.42,0.595) {\small $c_3$};
 	\node[left] at (0.42,0.7) {\small $c_1$};
 	\end{tikzpicture}
 	
    \caption{The graph of the map $T^2_s|_{[c_2, c_1]}$ where $s=1.4<\sqrt{2}$. Dashed lines denote the core of the renormalized maps {\allteal $T^2_s|_{[c_2, \pi]}$ and $T^2_s|_{[\pi, c_1]}$.}}
 	\label{fig:tents}
 \end{figure}

%
A last fundamental result we need is that there are enough tent maps to cover all types of chaotic dynamics of T-unimodal maps.
\begin{theoremX}[Milnor \& Thurston, 1977~\cite{MT88} (Thm. 7.4). See also Thm. 3.4.27 in~\cite{BB04}]\label{thm:C}
    Let $\cT$ be the innermost cyclic trapping region of a T-unimodal map $f$ and denote its period by $k$. 
    Then $f^k|_{J_1(\cT)}$ is T-unimodal and is topologically conjugated to a tent map $T_s$ for some $s\in(\sqrt{2},2]$.
\end{theoremX}
\begin{corollary}
\olive{
  Let $f$ be a T-unimodal map with two nodes and suppose that its attractor, as a trapping region, is not cyclic.
  Then $f$ is topologically conjugated to a tent map $T_s$ for some $s\in(\sqrt{2},2)$. 
  If $f$ has just one node, then it is topologically conjugated to $T_2$.
  }
\end{corollary}

Since there are parameter values $s$ for which the critical point of $T_s$ is periodic, there are tent maps for which the dynamics within the attractor is not of S-unimodal type.
\allviolet
Notice that, on the other side, every tent map with $s>1$ is T-unimodal (see Thm.~B).
The example below shows that, though, not all T-unimodal maps are topologically conjugated to tent maps.
\begin{example}
  Let $\mu_0$ be any parameter value of the logistic map within a window of the bifurcation diagram. 
  Then the second node of the graph of $\ell_{\mu_0}$ is a Cantor set (see Fig.~2 in~\cite{DLY20} for concrete examples based on the period-3 window). 
  Assume that the attractor corresponding to $\mu_0$ is chaotic (this happens for a set of parameter values of positive measure within the window). 
  Then $\ell_{\mu_0}$ is T-unimodal but not topologically conjugated to a tent map, since the only repellors of the tent map are periodic orbits.
\end{example}
Notice, finally, that there are T-unimodal maps that are neither conjugated to a S-unimodal map neither to a tent map (see Ex.~\ref{ex:Tu}).
%
\section{Nodes and trapping regions of T-unimodal maps}
In this section we show that the qualitative dynamics of a T-unimodal map has the following structure. 
Its nodes are naturally ordered with respect to the maximum value the map achieves on them.
We denote nodes as $N_0,\dots,N_p$ using the ordering above.
For each repelling node $N_i$ there is a cyclic trapping region $\cT(N_i)$ that has in common with $N_i$ a periodic orbit and such that no point of $N_i$ lies in $\Jall(N_i)$.
For $0<i<p$, the node $N_i$ can be characterized as the set of all chain-recurrent points in $\Jall(N_{i-1})\setminus\Jall(N_i)$.
For each node $N_i$, $i=0,\dots,p-1$, its core $\cK(N_i)$ is a tight trapping region strictly contained in  $\Jall(N_i)$.
The attractor $A$ is always a tight trapping region. 
If it is not cyclic, then is the core of the last repelling node: $A=K(N_{p-1})$.
If it is cyclic, then the periodic orbit on its boundary is in common with a repelling Cantor set. 

\allblack

\medskip\noindent
{\bf Nodes and cyclic trapping regions.}
Let $N$ be a repelling node of a T-unimodal map $f$. 
%
\allteal
Since $N\subset X$ is closed and $c\not\in N$ (see Thm~A), $f|_N$ takes a maximum value $n<c_1$. 
We denote by $J_1(N)$ the set of all $x\in[a,b]$ with $f(x)\geq n$. 
At least one of the endpoints of $J_1(N)$ belongs to $N$ and no point of $N$ is closer to $c$ than that point.
%
%
%
%
\allblack
\begin{proposition}   
    In any T-unimodal map, there is a natural linear order among repelling nodes defined by
    $M>N\hbox{ iff }J_1(M)\supset J_1(N).$
\end{proposition}
\begin{proof}
    \allteal
    Let $M,N$ be distinct nodes of $f$. Then, since $M\cap N=\emptyset$ and $f$ is unimodal, the maxima of $f|_M$ and of $f|_N$ are different and either $J_1(M)\subset J_1(N)$ or $J_1(N)\subset J_1(M)$.
\end{proof}
\allteal
From now on, throughout the article we will denote by $p+1$ the number of nodes of a map $f$ and by $N_0,\dots,N_p$ the nodes themselves ordered in such a way that \olive{$N_p$ is the attractor and $N_i<N_j$ if and only if $i<j$ for all repelling nodes.} 

Note that, in principle, $p$ can be any natural integer including 0 or even infinity. 
For instance, in case of the tent map $T_s$, we have $p=0$ if and only if $s\in(0,1)$ or $s=2$, $p=1$ if and only if $s\in[\sqrt{2},2)$ and so on.
Recall that the case $p=\infty$ happens if and only if the attractor is a Cantor set (see Thm~A). 

\allblack
%

%
%
\begin{lemma}
    \label{lemma:downstream}
    Let $N$ be a repelling node of a T-unimodal map $f$.
    Then:
    \begin{enumerate}
        \item every chain-recurrent point in $J_1(N)$ is downstream from $N$;
        \item $J_1(N)$ is forward-invariant by some positive power of $f$;
        \item one of the endpoints of $J_1(N)$ is periodic.
    \end{enumerate}
\end{lemma}
\begin{proof}
    (1) Let $p_0$ be an endpoint of $J_1(N)$ belonging to $N$, so that $f|_N$ takes its maximum value at $p_0$, and assume, for discussion sake, that $p_0<c$ (the argument for $p_0>c$ is virtually the same).
    Notice that no point of $N$ is closer to $c$ than $p_0$ since, at the left of $c$, $f$ is monotonically increasing.  Set $K=K_\varepsilon=[p_0,p_0+\varepsilon]$. 
    Since $f$ has no homtervals, for every $\varepsilon>0$ there is some integer $r>1$ such that $f^r(K_\varepsilon)$ contains $c$ in its interior.

    Since $p_0$ is the closest point of $N$ to $c$, $f^r(p_0)$ cannot be in the interior of $\J_1(N)$. Then $f^r(K_\varepsilon)$ must contain at least either $[p_0,c]$ or $[c,\hat p_0]$. 
    Denote by $\J_0$ the one of the two intervals above that $K_\varepsilon$ eventually maps onto for arbitrarily small $\varepsilon$.
    Hence, for arbitrarily small $\varepsilon$, there are $\varepsilon$-chains from $p_0$ to each point within $\J_0$. Notice that chain-recurrent points $x$ and $\hat x$ belong to the same node and that, if $p_0$ is upstream from any point of a node, it is upstream from each point of that node.
    
    \medskip\noindent
    (2) Let $\J_0=[p_0,c]$ and denote by $\inter{J}_1$ the interior of $\J_1(N)$.
    Since $f$ has no homtervals, there is some integer $r\geq1$ such that $f^r(\J_0)=f^r(\J_1(N))$ contains $c$ in its interior. In other words, there is a point $x$ in $\inter{J}_1$ such that $f^r(x)=c\in\inter{J}_1$.
    We claim that $\inter{J}_1$ is invariant under $f^r$. 
    Indeed, by continuity, if there were a $y\in\inter{J}_1$ with $f^r(y)\not\in\inter{J}_1$, there would be some $\xi\in\inter{J}_1$ between $x$ and $y$ such that either $f^r(\xi)=p_0$ or $f^r(\xi)=\hat p_0$. 
    This, though, is impossible because, \teal{by point 1 above},  there are no preiterates of $N$ in $\inter{J}_1$.
    
    \medskip\noindent
    (3) By continuity, the common value of $p_0$ and $\hat p_0$ under $f^r$ must belong to $\J_1(N)$ and, since $N$ is forward invariant and there are no points of $N$ in $\inter{J}_1$, it can only be either $p_0$ or $\hat p_0$, namely either $f^r(p_0)=p_0$
    or $f^r(\hat p_0)=\hat p_0$,
    namely one of the two endpoints is periodic.
    
\end{proof}
From now on, we will denote by $p_1(N)$ the periodic point of $N$ on which the maximum value of $f|_N$ is achieved. Recall that the endpoints of the interval $J_1(N)$ are $p_1(N)$ and its conjugate.
\begin{proposition}
    \label{prop:repellTrap}
    Let $f$ be a T-unimodal map. 
    There is a natural bijection $N\to\cT(N)$ between the repelling nodes of $f$ and its cyclic trapping regions that
    satifies the following properties:
    \begin{enumerate}
                \item $J_1(\cT(N))=J_1(N)$;
                \item $|\cT(N)|$ is equal to either the period of $p_1(N)$ or to its double.
    \end{enumerate} 
    The inverse of this map sends a cyclic trapping region $\cT$ into the node containing $\gamma(\cT)$.
\end{proposition}
\begin{proof}
    Let $r$ be the smallest integer such that $f^{r}(J_1(N))\subset J_1(N)$ -- this number exists because of the Lemma above -- and recall that at least one of the two endpoints belongs to $N$, and therefore all of its iterates do.
    Notice that there can be no point upstream from $N$ in the interior of any of the intervals $J_{k+1}=f^k(J_1(N))$.  Indeed no point upstream from $N$ can be in the interior of $J_1(N)$ by the Lemma above and the same holds for $J_2=[f(p_1(N)),c_1]$, since $f(p_1(N))=f(\hat p_1(N))$, and so for every other $J_k$, since the restriction of $f$ to any interval in the complement of $J_1(N)$ is a homeomorphism. 
    Hence the interiors of the $J_k$ are all pairwise disjoint, namely $\cT(N)=\{J_1(N),J_2,\dots,J_{r-1}\}$ is a trapping region. Moreover, by the Lemma above, the endpoint $p_1(N)$ of $J_1(N)$ is periodic and so $\cT(N)$ is cyclic. 
    
    Now, denote by $k$ the period of $p_1(N)$. Then $f^k(J_1(N))$ must have at least $p_1(N)$ in common with $J_1(N)$. If these two intervals are on the same side of $p_1(N)$, then $r=k$. If they are on opposite sides, then $r=2k$.
\end{proof}
%
%
\allteal

\begin{proposition}
  \label{prop:unimodal}
  Let $f$ be T-unimodal map and $N$ a repelling node of $f$. Then the restriction of $f^r$ to $J_1(N)$ is T-unimodal, where $r\geq 1$ is as in Proposition~\ref{prop:repellTrap}.
\end{proposition}
\begin{proof}
    Denote by $J_1=J_1(N)$, and $J_i=f^{i-1}(J_1)$, $2\leq i<r$. The proof of Proposition~\ref{prop:repellTrap} implies that $f^r(J_1)\subset J_1$, and $J_1, \ldots, J_{r-1}$ are disjoint. 
    Note first that $f^r|_{J_1}$ is unimodal since $J_1$ contains $c$ in its interior, and both endpoints $p_1=p_1(N)$ and $\hat p_1$ are mapped to $p_1$ (recall the notation of Proposition~\ref{prop:repellTrap}). We show that it is $T$-unimodal.
    
    We first argue that a wandering interval of $f^r|_{J_1}$ is also a wandering interval of $f$. Assume that $J\subset J_1$ is a closed interval such that $f^{rk}(J)$ are mutually disjoint for all $k\geq 0$. Since $J_1, \ldots, J_{r-1}$ are mutually disjoint, it follows that $f^k(J)$ are mutually disjoint for all $k\geq 0$. Furthermore, if $p\in J_1$ is an $n$-periodic orbit of $f$, then $n$ is a multiple of $r$, and $p$ is an $n/r$-periodic orbit of $f^r$. Thus, if an interval $J\subset J_1$ is $f$-attracted to a periodic orbit of $f$, then it is also $f^r$-attracted to a periodic orbit of $f^r$. That implies that a wandering interval of $f^r|_{J_1}$ is also a wandering interval of $f$. Since $f$ is $T$-unimodal, and thus has no wandering intervals, we conclude that $f^r|_{J_1}$ has no wandering intervals. 
    
    Furthermore, an attractive $m$-cycle of $f^r|_{J_1}$ is an attractive $mr$-cycle of $f$. Since $f$ is $T$-unimodal, $f^r|_{J_1}$ does not have attracting cycles.
    
    Finally, note that every chain-recurrent point of $f$ in $J_1$ is a chain-recurrent point of $f^r|_{J_1}$. Furthermore, if $x,y\in J_1$ are chain-recurrent points of $f$, and $x$ is $f$-downstream from $y$, then $x$ is also $f^r$-downstream from $y$. So, if $x, y\in J_1$ are in the same node of $f$, then they are also in the same node of $f^r|_{J_1}$. Thus, if $f^r|_{J_1}$ would have infinitely many nodes, so would $f$, which is a contradiction with $f$ being $T$-unimodal. It follows that $f^r|_{J_1}$ is $T$-unimodal.
\end{proof}
\begin{proposition}
  \label{prop:surjective}
  Let $f$ be T-unimodal map such that the boundary fixed point $a$ is repelling. 
  
  Then: 
  %
  \begin{enumerate}
      \item there is no chain-recurrent point in $(a,c_2)$;
      \item either $p=0$ and $N_0=[a,b]$ or $p>0$ and $N_0=\{a\}$; 
      \item if $p=1$, then $N_1=[c_2,c_1]$. 
  \end{enumerate}
\end{proposition}
 \begin{proof}
       {\bf 1.} We assume that $c_1 < b$ or there is nothing to prove. Notice that, since by hypothesis $f$ has no attracting periodic orbits, by Prop.~\ref{prop:homtervals} $f$ has no homtervals. Hence, for any $x_0\in(a,c_2)$, there is some $k>0$ such that $c\in f^k([a,x_0])$, namely $f^k(x_0)\in[c_2,c_1]$.
       Recall from Corollary~\ref{cor:Tunim} that $c$ belongs to the attractor and that almost all points of $[a,b]$ are in the basin of attraction of the attractor.
       Hence $c\geq c_2$, since $c$ belongs to the attractor and the attractor lies necessarily in $[c_2,c_1]$ because that interval is a trapping region by Proposition~\ref{prop:core}.
       This means that every point in $(a,c_2)$ falls into $[c_2,c_1]$ in a finite number of iterations.
       By continuity, therefore, for $\varepsilon$ small enough, every $\varepsilon$-chain based at $x_0$ falls into $[c_2,c_1]$ in a finite number of iterations.
       Moreover, since $c_3>c_2$, every point close enough to $[c_2,c_1]$ falls into it in either two steps (if it lies at its right) or one step (if it lies on its left).
       Hence, for $\varepsilon$ small enough, every $\varepsilon$-chain based at a point in $[c_2,c_1]$ cannot get further than some $d(\varepsilon)$ from $[c_2,c_1]$ with $d(\varepsilon)\to0$ for $\varepsilon\to0$.
       The two observations above imply that no $\varepsilon$-chain based at $x_0\in(a,c_2)$ can get back at $x_0$ for $\varepsilon$ small enough.
         
       {\bf 2.} If $p=0$, namely there is a single node $N_0$, then this node must be the attractor. 
       Since the core $[c_2,c_1]$ contains the attractor and $x=a$ is chain-recurrent, the only possibility is that $f^2(c)=a$ and therefore that $f(c)=b$.
       In particular, $f$ is surjective.
         
      If $p>0$, then it must be that $c_2>a$ or, by the argument above, there would be a single node. 
      Hence, there is no chain-recurrent point close enough to $a$ and so $\{a\}$ is a node by itself.
      
      {\bf 3.} If $p=1$, then there are no other nodes besides the boundary fixed point $a$ and the attractor $N_1$. 
      By the same argument above, the attractor must be a connected interval.
      Since $c\in N_1$, then $c_1,c_2\in N_1$ and so $N_1=[c_2,c_1]$. 
 \end{proof}
\begin{example}\label{ex:tent1}
  In case of the tent map family, $0$ is repelling for $s\in(1,2]$. Then we have $p=0$ 
  for $s=2$,  $p>0$ for $s\in[1,2)$ and $p=1$ for $s\in[\sqrt{2},2)$.
\end{example}
\begin{corollary}
  Let $f$ be a surjective T-unimodal map.
  Then $f$ has the single node $N_0=[a,b]$. 
\end{corollary}

\begin{lemma}
    \label{lemma:notsure}
  Let $N$ be a repelling node of a T-unimodal map $f$ and set $k=|\cT(N)|$. Then the restriction of $f^k$ to $J_1(N)$ is not surjective.
\end{lemma}
\begin{proof}
    If the restriction of $f^k$ to $J_1(N)$ were surjective then, by Prop.~\ref{prop:surjective}, the trapping region $\cT(N)$ itself would be a chaotic attractor. In this case, though, there would be an intersection between an attracting node and a repelling node, which is impossible.
\end{proof}
\allblack
\begin{proposition}
    \label{prop:characterization}
    Let $M,N$ be two distinct repelling nodes of $f$ with $M>N$.
    Then:
    \begin{enumerate}
        \item  $|\cT(N)|>|\cT(M)|$ and $|\cT(M)|$ divides $|\cT(N)|$;
        \item $N\subset\Jall(M)\setminus\Jall(N)$;
        \item if $M$ and $N$ are consecutive, $N$ is the set of all chain-recurrent points in $\Jall(M)\setminus\Jall(N)$.
    \end{enumerate}
\end{proposition}
\allteal
\begin{proof}
    {\bf 1.} Set $\cT(M)=\{\J_1,\dots,\J_k\}$ and
    $\cT(N)=\{\J'_1,\dots,\J'_{k'}\}$.
    Since the maps $f:\J_i\to \J_{i+1}$,
    $i=2,\dots,k$ are all homeomorphisms, where we set $\J_{k+1}=\J_1$, then inside each $\J_i$ there must be the same number of intervals $\J'_i$, namely $k'=mk$ for some integer $m\geq1$.
    
    Suppose now that $k'=k$ and assume, for discussion sake, that $p_1(M)<c$.
    Then, since $f^k(\J'_1)\subset \J'_1$, both $p_1(M)$ and $p_1(N)$ are fixed points for $f^k$. 
    
    If $p_1(N)<c$ as well, then $N$ would be an attracting node. Indeed, if it were not, since $f^k$ is increasing between $p_1(M)$ and $p_1(N)$, and those points are both repelling by assumption, there would be an attracting point $x$, belonging to an attracting node $N'$, between $p_1(M)$ and $p_1(N)$. 
    Hence $N$ would be inside $\Jall(N')$ but, since $N'$ is an attractor, $\Jall(N')$ is a subset of the basin of $N'$.
    
    If, on the contrary, $p_1(N)>c$, then 
    \magenta{$f^k$ would be decreasing nearby $p_1(N)$ and so  $f^k(J'_1)$ would be on the other side of $J'_1$ with respect to $p_1(N)$, so that $\J'_1$ could not be invariant under $f^k$.} 
    

    {\bf 2.} Since $M>N$, there is at least one point of $N$ inside $\Jall(M)$ and therefore, since both $\Jall(M)$ and $N$ are forward-invariant, the whole $N$ must be contained in $\Jall(M)$. 
    Moreover, by construction, no point of $N$ can fall onto $\Jall(N)$.
    
    {\bf 3.} Suppose that there is a chain-recurrent point $x\in\Jall(M)\setminus\Jall(N)$ not belonging to $N$. Denote by $N'$ the node $x$ belongs to.
    Without loss of generality we can assume that $x\in \J_1(M)$. Notice that this means that $M>N'$.
    By Prop.~\ref{prop:upstream}, $N'$ cannot have points in $\Jall(N)$ since no chain-recurrent point of that set can be upstream to any chain-recurrent point outside of it. Hence, $N'>N$, which contradicts the hypothesis that $M,N$ are consecutive.
\end{proof}
\begin{definition}
    \label{def:Ncore}
    For $k = 0,\dots,p-1$, we denote by $r_k$ the period of $\cT (N_k)$ and set
    $$\bar f_{i} = f^{r_k}|_{J_i(N_k)} : J_i(N_k) \to J_i(N_k), i = 1,\dots,r_k.$$
    We denote by $K(\bar f_{i})$ the core of the T-unimodal map $\bar f_{i}$ and call {\bf core of $N_k$} the collection of intervals
    $$\cK(N_k) = \{K(\bar f_{1}),\dots,K(\bar f_{r_k})\}.$$
    Finally, we denote by $K(N_k)$ the union of all intervals in $\cK(N_k)$.
\end{definition}
\allmagenta
\begin{example}\label{ex:tent2}
    When the endpoint $a$ is repelling and $c_2>a$, $N_0=\{a\}$ and $\cK(N_0)=\{[c_2,c_1]\}$. This is the case of the tent map $T_s$ for $s\in(1,2)$, as it is clearly illustrated in Fig.~\ref{fig:tm}. 
    When the inner fixed point $\bp$ is repelling and a node in itself, then $N_1=\{\bp\}$ and $\cK(N_1)=\{[c_2,c_4],[c_3,c_1]\}$. 
    In case of the tent map, this holds for $s\in(1,\sqrt{2})$, as it is evident from Fig.~\ref{fig:tm}.
\end{example}
\begin{example}
  Consider a parameter $\mu$ belonging to the period-3 window of the bifurcation diagram in Fig.~\ref{fig:Tunimodal2}. Then $N_0=\{a\}$ and the second node $N_1$ is a Cantor set.
  The corresponding cyclic trapping region $\cT(N_1)$ has period equal to 3 and $\cK(N_1)=\{[c_2,c_5],[c_3,c_6],[c_4,c_1]\}$. 
\end{example}

\begin{proposition}\label{prop:cores}
    The following properties hold for each $k=0,1,\dots,p-1$:
    \begin{enumerate}
    \item $K(\bar f_{r_k-i}) = [c_{2r_k-i-1},c_{r_k-i-1}]$ for $i=-1,0,\dots,r_k-2$, where $\bar f_{r_k+1}=\bar f_1$; 
    \item $K(N_k)\subset\Jall(N_k)$;
    \item for $k\leq p-2$, $\cK(N_k)$ is a trapping region and $\Jall(N_{k+1})\subset K(N_k)$.
  \end{enumerate}
\end{proposition}
\begin{proof}\
  Recall that $J_1(N_k)=[p_1(N_k),\hat p_1(N_k)]$ is the interval of $\cT(N_k)$ containing $c$, so that
  $$K(\bar f_1) = [\bar f_1^2(c),\bar f_1(c)] = [c_{2r_k},c_{r_k}].$$
  Moreover, each map $f|_{J_i(N_k)}:J_i(N_k)\to J_{i+1}(N_k)$, $i=2,\dots,r_k$, where we set $J_{r_k+1}(N_k)=J_1(N_k)$,
  is a homeomorphism for $i>1$, and so $K(\bar f_{r_k}) = f^{-1}(K(\bar f_1))\cap J_{r_k}(N_k)=[c_{2r_k-1},c_{r_k-1}]$ and, ultimately,
  $$
  K(\bar f_{r_k-i}) = [c_{2r_k-i-1},c_{r_k-i-1}].
     $$
  The assumption that $N_k$ is not the last node implies the following two facts:
\begin{enumerate}
  \item
  $K(\bar f_{1})\subset int(J_1(N_k))$. Otherwise, $\cT(N_k)$ would be a tight trapping region and therefore an attractor and $N_k$ would be the last node, against the fact that there are $p+1$ nodes. Hence, in general,
  $K(\bar f_{i})\subset int(J_{i}(N_k))$ for $i=1,\dots,r_k$, which proves point (2).
  \item
  Since the cycle $\Gamma(N_k)$ is repelling and, for $k<p-1$,  $N_{k+1}$ is
  repelling too, then by Proposition~\ref{prop:core} we have that $K(\bar f_i)$
  is forward-invariant with respect to $\bar f_i$, so that
  $$
  f(K(\bar f_1)) = f^2(K(\bar f_{r_k})) = \dots = f^{r_k}(K(\bar f_2)) \subset  K(\bar f_2)
  $$                                      and, more generally, $f(K(\bar f_i))\subset f(K(\bar f_{i+1}))$, $i=1,\dots,r_k$, where we put $\bar f_{r_k+1}=\bar f_1$. 
\end{enumerate}                                                                
  Since the attractor is unique, it must lie inside $K(N_k)$ and so $\Jall(N_{k+1})$ must be completely  
  contained inside $K(N_k)$ as well. In particular, $c\in J_1(N_{k+1})\subset K(\bar f_1)$, so $\cK(N_k)$ is a trapping region, proving point (3).
\end{proof}

\allblack

\section{Classification of nodes and main results}
With the results above, we can classify all types of nodes of T-unimodal maps.
\begin{definition}
    Let $N$ be a repelling node $f$. 
    We say that $\cT(N)$ is {\bf regular} (resp. {\bf flip}) if, in a small enough neighborhood of any $p\in\gamma(\cT(N))$, $f$ is increasing (resp. decreasing).
\end{definition}
\begin{proposition}
  A repelling node for a T-unimodal map $f$ is either a repelling periodic orbit or a Cantor set on which $f$ acts transitively. 
  In particular, each point of a repelling node is also a non-wandering point.
\end{proposition}
\begin{proof}\
    By Prop.~\ref{prop:characterization}, $N_{i+1}$ 
    is the set of chain-recurrent points in $\Jall(N_i)\setminus\Jall(N_{i+1})$.
    Denote by $\bar f$ the power $|\cT(N_i)|$ of $f$ and let $\cS$ be the collection of all intervals of $\cT(N_{i+1})$ that lie in $J_1(N_i)$. 
    Notice that $\cS$ is a cyclic trapping region of period $k=|\cT(N_{i+1})|/|\cT(N_i)|$ for the restriction of $\bar f$ to $J_1(N_i)$ and that $S_1=J_1(\cS)=J_1(N_{i+1})$.
    We can always reduce the problem of the structure of the set $N_{i+1}$ to the problem of the structure of the set of points of the single interval $\J_1(N_i)$ not falling, under $\bar f$, on $\cS$. 
    
    Assume first that
    $\cS$
    is regular. Then, in particular:
    \begin{enumerate}
        \item $S_i\cap S_j=\emptyset$ for $i\neq j$;
        \item for all $i\neq 1$, the set $\bar f^{-1}(S_{i+1})$, where we use the notation $S_{k+1}=S_1$, is the disjoint union of $S_i$ and a second interval $S'_i$, on each of which $\bar f$ restricts to a homeomorphism 
        with $S_{i+1}$ 
        \item $\bar f^{-1}(S_{2})=S_1$. 
        In this case, every point of $f(S_1)\subset S_2$
        is covered by two points of $S_1$, 
        except for $f(c)$. 
    \end{enumerate}
    Now, take any two intervals $A$ and $B$ that are connected components of, respectively, $\bar f^{-p}(S_i)$ and $\bar f^{-q}(S_j)$ and 
    assume, in order to avoid trivial cases, that neither $A$ nor $B$ are equal to $S_1$. 
    Then there is $r>0$ such that $\bar f^r(A)$ and $\bar f^r(B)$ are subsets, respectively, of two intervals $S_{i_A}$ and $S_{i_B}$ and at least either $\bar f^r(A)$ or $\bar f^r(B)$ is equal to that interval.
    
    Assume first that $i_A\neq i_B$. Then $\bar f^r(A)\cap\bar f^r(B)=\emptyset$ and therefore also $A\cap B=\emptyset$.
    Assume now that $i_A=i_B=i$;
    then $A\cap B\neq\emptyset\implies A=B$.
    If it is not so, then there are two possibilities. 
    The first is that $A\cap B$ is a single point.
    In this case, though, there would be a $s$ such that $S_i\cap\bar f^{sk}(S_i)$ is a single point, but this can happen only if $\cS$ is flip, against the hypothesis.
    The second possibility is that $A\cap B$ contains an open set. 
    In this case, either $A=B$ or there is an endpoint of one of the two intervals, say $A$ in the interior of the other one. 
    This last possibility, though, leads to a contradiction. 
    Indeed, let $s>0$ be such that $\bar f^s(A)=S_1$
    and recall that $\bar f^s|_{A}$ is a homeomorphism and so brings endpoints to endpoints. 
    Then, in the interior of $S_1$ there would be a point that eventually falls on $N_i$, but this is impossible because then such point would be at the same time upstream and downstream from $N_i$ and so would belong to it.    
    

    \allblack
    
    Ultimately, therefore, the set of points of $J_1(N_i)$ that never fall in $\Jall(\cS)$ is the complement of a countable dense set of open intervals whose closures are pairwise disjoint.
    Hence, it is a Cantor set and the action of $\bar f$ on it is a subshift of finite type. 
    The node $N_{i+1}$ is a closed invariant subset of a finite union of such Cantor sets, and therefore is itself a Cantor set.
    

    
    Assume now that $\cS$ is a flip cyclic trapping region. Then $\cS=\{S_1,S_2\}$, with $S_1=J_1(N_{i+1})$.
    Hence, $\cS$ is as in Fig.~\ref{fig:tr}, namely $S_1=[q_1,p_1]$ and $S_2=[p_1,q_2]$, where $p_1=p_1(N_{i+1})$ is fixed for $\bar f$.
    Then $\bar f^{-1}(S_1)$ is the disjoint union of $S_2$ and the interval $A=[\tilde q_1,q_1]$, where $\tilde q_1$ is the counterimage of $q_1$ at the left of $c$, while 
    $\bar{f}^{-1}(S_2)=S_1$. 
    The two counterimages of $A$ are the intervals $A_1=[\doubletilde{q}_1,\tilde q_1]$ and $A_2=[\bar q_1,\bar{\bar{q}}_1]$, where $\doubletilde{q}_1$ is the counterimage of $\tilde q_1$ at the left of $c$, $\bar{\bar{q}}_1$ the one at the right of $c$ and $\bar{q}_1=q_2$ is the counterimage of $q_1$ at the right of $c$. 
    Similarly, at every new recursion step, two new intervals arise, one at the left of $c$ and having an endpoint in common with the interval at the left of $c$ obtained at the previous recursion level and one at the right with similar properties. 
    
    Ultimately, then, the set of points of $J_1(N_{i})$ that do not fall eventually in $\Jall(N_{i+1})$ under $\bar f$ is the union of the fixed point $p_1(N_i)$ together with all of its counterimages under $\bar f$. These counterimages can be sorted into two subsequences which converge monotonically to the endpoints of $J_1(N_i)$. Hence, in this case $N_{i+1}$ consists exactly in the flip periodic orbit through $p_1(N_{i+1})$.
\end{proof}
%
%
\begin{proposition}
  \label{prop:A}
  Let $f$ be a T-unimodal map with attractor $A$ and attracting node $N_p$.
  Then $N_p$ can be of the following two types:
  \begin{enumerate}
      \item if $A$ does not intersects any repelling Cantor set, then $N_p=A$ (case $A_2$ in~\cite{DLY20});
      \item if $A$ does intersect a repelling Cantor set $C$ (notice that, in this case, $A\cap C$ must be a periodic orbit at the boundary of $A$), then $N_p$ is a trapping region containing $A$, $C$ and part of the basin of attraction of $A$ (case $A_5$ in~\cite{DLY20}).
  \end{enumerate}
  In particular, in the first case, $\Rf=\Omega_f$ while, in the second case, $\Rf\supsetneq\Omega_f$.
\end{proposition}
\begin{proof}\ 
\allmagenta
    Suppose first that the attractor $A$ is isolated from the rest of the non-wandering set, namely that there is an open neighborhood of $A$ that contains no other non-wandering point. 
    Then an argument of the same kind used to prove Prop.~\ref{prop:surjective} shows that, for each $x\in A$, there is no $y$ outside of $A$ such that, for every $\varepsilon>0$, there is an $\varepsilon$-chain from $x$ to $y$. 
    In particular, this means that $N_p=A$.
    
    Consider now the case when $A$ is not isolated. This can happen only under the following circumstances:  
    $A$ is a cyclic tight trapping region and 
    the periodic orbit at its boundary also belongs to a a repelling Cantor set $C$. In this case, $N_p$ is equal to $K(N_{p-1})$, that contains $A$, $C$ and part of the basin of attraction of $A$.
\end{proof}
Notice that the tent map has no Cantor repelling nodes (see Thm.~B) and so its non-wandering set always coincides with its chain-recurrent set. 
\allblack
\begin{proposition}
    \label{prop:upstream}
    If $N$ and $M$ are nodes with $N>M$, then there is an edge from $N$ to $M$. 
\end{proposition}
\begin{proof}
    By construction, the node $N$ has no common point with $\Jall(\cT(N))$. Since $M$ is closer to $c$ than $N$, on the contrary, at least one of its points lies in the interior of $\J_1(N)$ and therefore the whole $M$ lies in $\Jall(\cT(N))$.

    By Lemma.~\ref{lemma:downstream}, each point $x\in M$ is downstream from $p_1(N)$. Since $p_1(N)$ is periodic, for every $\varepsilon>0$ there is a trajectory $t_\varepsilon$ starting in $(p_1,p_1+\varepsilon)$ and falling eventually on $x$. Since $p_1$ is repelling, for each point $y$ close enough to $p_1$ there is a trajectory $t$ passing through $y$ whose backward limit is contained in $N$. 
    In other words, there are trajectories backward asymptotic to $N$ from any node inside $\cT(N)$, namely there is an edge from $N$ to any node in $\Jall(N)$.
\end{proof}
\begin{definition}
    We say that an acyclic directed \gr graph $\Gamma$ is a {\bf tower} if there is an edge between every pair of distinct nodes of $\Gamma$.
\end{definition}
\begin{theorem}
    \label{thm:tower}
    The graph of any T-unimodal map is a finite tower.
\end{theorem}
\begin{proof}
    This is an immediate consequence of Assumption (T) and Prop.~\ref{prop:upstream}.
\end{proof}
%

\allviolet\medskip
\begin{figure}
 \centering
 \includegraphics[width=12cm]{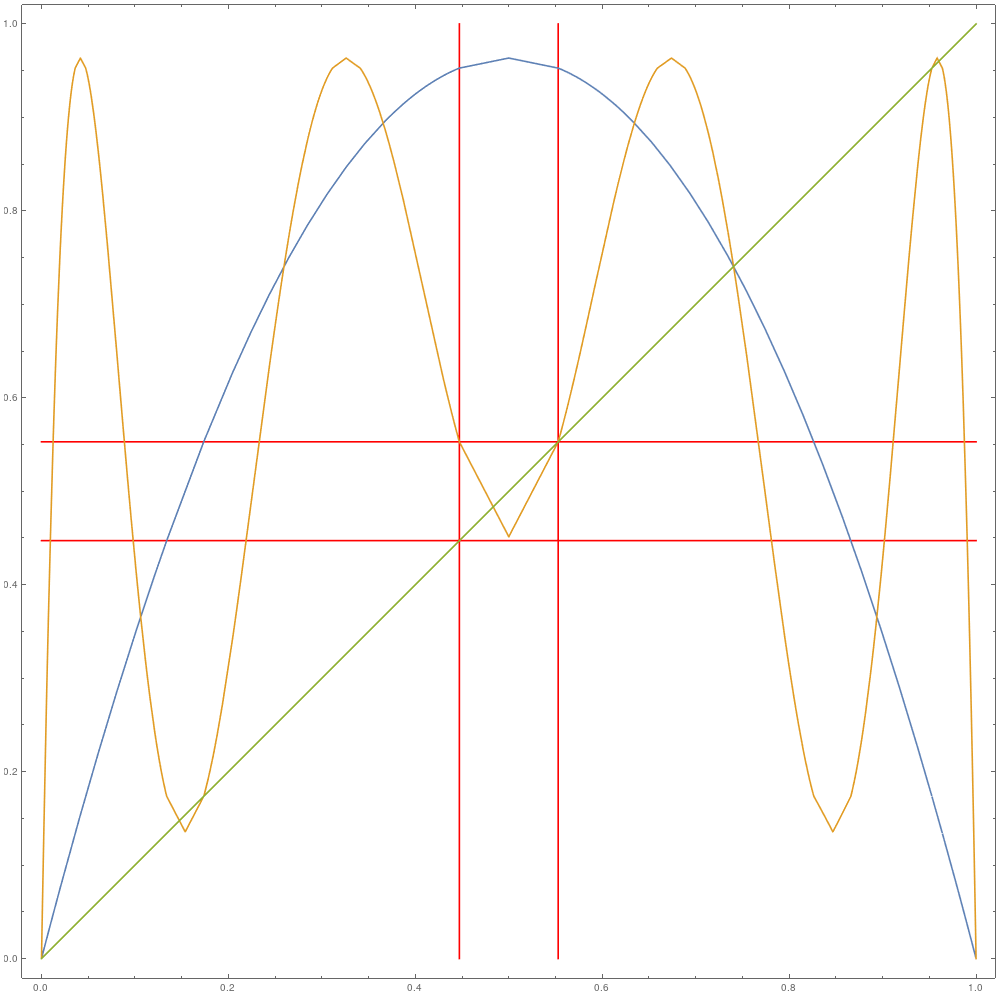}
 \caption{{\em Graph of the function $u_1$ from Ex.~\ref{ex:Tu} (blue) and its cube (orange).} 
}
 \label{fig:Tunimodal1}
\end{figure} 
\begin{figure}
 \centering
 \includegraphics[width=14cm]{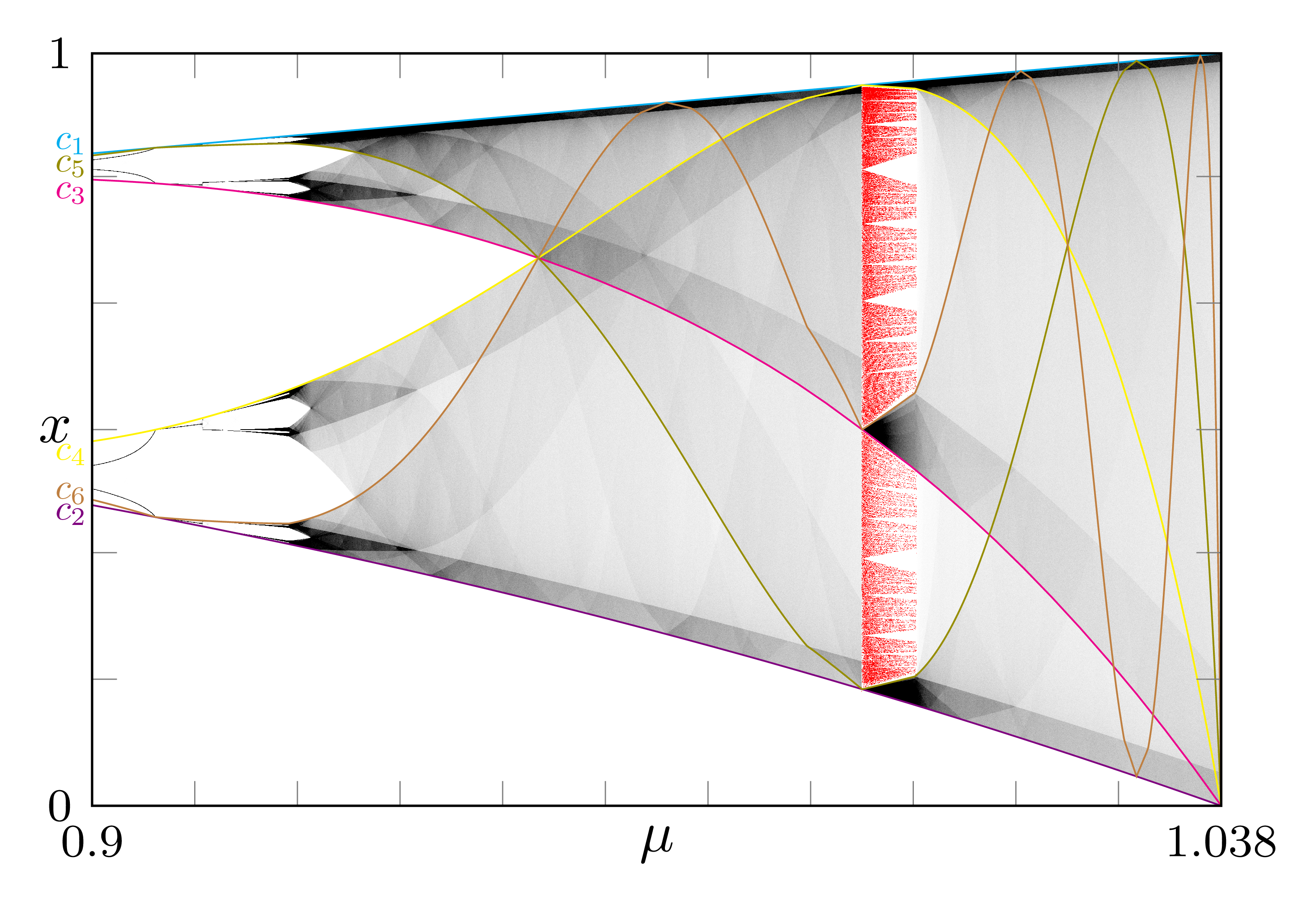}
 \caption{{\em Bifurcation diagram of a T-unimodal family.} 
 This picture shows the bifurcation diagram of the family $u_\mu(x)=\mu F(x)$, where $F(x)$ is the piecewise polynomial map in Ex.~\ref{ex:Tu}. This family leaves $[0,1]$ invariant for $\mu\in[0,4/3.854]$.
 Attractors are painted in shades of gray (depending on the density), repelling periodic orbits in green and repelling Cantor sets in red. 
 The colored lines labeled by $c_k$ are the lines $u^k_\mu(c)$, where $c = 0.5$ is the critical point of $u_\mu$. 
 This map has a single window, of period 3, for approximatively $0.994\leq\mu\leq1.001$. 
 At $\mu=1$, the repelling Cantor set coincides, by construction, with the one of the logistic map at $\mu=3.845$.
}
 \label{fig:Tunimodal2}
\end{figure} 
One can use the results above to build T-unimodal maps that are not topologically conjugated to either tent maps or S-unimodal maps. 
\begin{example}
  \label{ex:Tu}
  Let $f$ be any logistic map in the interior of the period-3 window. 
  For each such a map, the second node is a Cantor set $N$.
  The point $p_1=p(N)$ has period 3. We set $p_2=f(p_1)$ and $p_3=f(p_2)$.
  The maximal cyclic trapping region of $N$ is 
  $$
  J_1=[q_1,p_1],\;J_2=[p_2,q_2],\;J_3=[q_3,p_3],
  $$
  where 
  $$
  q_1=1-p_1,\;f(q_1)=p_2,\;f(q_2)=q_3,\;f(q_3)=q_1.
  $$
  Any unimodal map $F$ that coincides with $f$ outside $J_1\cup J_2\cup J_3$ has $N$ as its second node. 
  In particular, such $F$ cannot be topologically conjugated to a tent map.
  
  Let us now consider the following $F$. In $J_2$ and $J_3$, $F$ is linear.
  In $J_1$, $F$ is a symmetric tent map.
  This way, the attractor of $F$ is chaotic as long as $F(c)$ is chosen large enough. 
  Moreover, we can always choose the value of $F(c)$ so that $c$ is periodic. Then $F$ then cannot be topologically conjugated to a S-unimodal map either. 
  
  In Fig.~\ref{fig:Tunimodal2}, we show the bifurcation diagram of the family $u_\mu(x)=\mu F(x)$, where $F$ is built as described above for $f(x)=3.854x(1-x)$ (see Fig.~8 in~\cite{DLY20} for the graph of $f$ and of $f^3$ and the trapping region of $N$) and we set $F(c)=f(c)$. Arbitrarily close to $F$, there are maps whose critical point is periodic. 
\end{example}
\allblack

\medskip
In case of the tent map, we get the following more specific result.
\begin{theorem}
    \label{thm:ttower}
      Let $s$ such that $\log_2 s\in [2^{-p},2^{1-p})$ for some integer $p\geq1$. Then $\Gamma_{T_s}$ is a finite tower with $p+1$ nodes $N_0,\dots,N_p$, where:
      \begin{enumerate}
        \item $N_0$ is the fixed boundary point;
        \item the subgraph $\{N_1,\dots,N_{p-1}\}$ is a $(p-1)$-cascade; 
        \item $N_p=K(N_{p-1})$, namely $N_p$ is the core of $N_{p-1}$.
      \end{enumerate}
\end{theorem}
\begin{proof}
Recall that the non-wandering set of $T_s$ coincides with its chain-recurrent set, and its structure is given in Theorem~\ref{thm:B}. Since $p\geq 1$, the endpoint $0$ is repelling, and thus $N_0=\{0\}$, see Example~\ref{ex:tent2}. Proposition~\ref{prop:cores} implies that all other nodes are contained in $K(N_0)=[c_2,c_1]$, and are downstream from $N_0$. If $p=1$, $N_1=[c_2,c_1]$ is a chaotic attractor of period $1$.\\
If $p>1$, recall from proof of Theorem~\ref{thm:B} that the core decomposes as $[c_2,c_1]=[c_2,\pi]\cup[\pi,c_1]=J_1\cup J_2$, where $T_s(J_1)=J_2$, $T_s(J_2)=J_1$, and $T^2_s|_{J_1}$ and $T^2_s|_{J_2}$ are conjugate to $T_{s^2}|_{[0,c_1]}$ with repelling fixed point $\pi$. The rest of the proof follows inductively, since $\log_2 s^2\in [2^{-p+1},2^{2-p})$.
\end{proof}

\medskip\noindent
{\bf Backward asymptotics.} We now have all tools and results to find $s\alpha$-limits of points in T-unimodal maps.
\allmagenta
\begin{definition}
    A {\bf backward orbit} based at $x$ is a sequence $\{\dots,x_{-2},x_{-1},x_0\}$
    such that $f(x_{-i-1})=x_{-i}$, $i=0,1,\dots$, and $x_0=x$.
\end{definition}
\begin{theoremX}[\cite{Mal12}]
  \label{thm:trans}
  Each transitive map $f$ has a dense set of points with a dense backward orbit.
\end{theoremX}
\begin{theorem}
    \label{thm:dense}
    Let $f$ be a T-unimodal map and denote by $A$ its attractor. 
    Then each point $x\in A$ has a backward orbit dense in $A$.  
\end{theorem}
\begin{proof}
   Let $d=\{\dots,d_{-2},d_{-1},d_0\}$ be a  backward trajectory dense in $A$ -- such a trajectory exists because of Thm.~\ref{thm:trans}.
   Now, denote by $r$ the period of the innermost cyclic trapping region $\cT$ of $f$. By Thm.~\ref{thm:C}, $\bar f=f^r|_{J_1(\cT)}$ is topologically conjugated to a tent map with a chaotic attractor consisting of a single interval $C$.
   In particular, the restriction of $\bar f$ to $C$ is topologically exact.
   Hence, given any two points $x,y\in C$, for every $\varepsilon>0$ we can find a $x'$ such that $|x-x'|<\varepsilon$ and $y$ lies on the orbit of $x'$ under $\bar f$.
   Take now any point $x\in C$.
   By the argument above, since $d$ is dense in the whole $A$, we can find a $d'_0$ closer than than $1/2$ to $d_0$ and whose orbit passes through $x$. 
   Then we can find a $d'_{-1}$ closer than $1/2^2$ to $d_{-1}$ whose orbit passes through $d'_0$.
   By proceeding this way, we obtain a backward orbit  $d'=\{\dots,d'_{-2},d'_{-1},d'_0,x\}$ under $\bar f$ based at $x$ such that $|d'_{-i}-d_{-i}|<2^{-i}$. Hence $d'$ is dense in $C$ and, therefore, the backward orbit of $x$ under $f$ is dense in $A$.
   By repeating the same argument on the other intervals $J_i(\cT)$, we get the claim.
\end{proof}
%
\allblack
\begin{definition}
  We define sets $U_1,\dots,U_p$ as follows:
  \begin{itemize}
  \item $U_p=N_p$;
  \item $U_{p-1}=K(N_{p-2})\setminus N_p$;
    \item $U_{k}=K(N_{k-1})\setminus K(N_{k})$ for $1\leq k<p-1$;
     
  \end{itemize}
  \noindent
  Finally, we set $U_0=[a,c_2)$ and $U_{-1}=(c_1,b]$. 
  \teal{For short, we say that $x$ is a level-$k$ point if $x\in U_k$.}
\end{definition}
The proof of the following proposition is a special case of the one of Prop.~6 in~\cite{DL22}.
\begin{proposition}
  The sets $U_{-1},\dots,U_p$ satisfy the following properties:
  \begin{enumerate}
  \item $[a,b]=\bigsqcup_{i=-1}^pU_i$;
  \item $N_k\subset U_k$ for $k=0,\dots,p$;
  \item $U_k\subset\Jall(N_{k-1})$ for $k=1,\dots,p$.
  \end{enumerate}
\end{proposition}
\begin{proof}
    1. By construction, the $U_i$ are all pairwise disjoint and, for T-unimodal maps, $p<\infty$. 
    Hence
    $$
    \bigcup_{i=1}^pU_i = (K(N_0)\setminus K(N_1)) \cup (K(N_1)\setminus K(N_2)) \cup \dots
    \cup (K(N_{p-2})\setminus N_p)\cup N_p =
    $$
    $$
    = K(N_0) = [c_2,c_1]
    $$
    and 
    $$
    \bigcup_{i=-1}^pU_i=[a,c_2)\cup[c_2,c_1]\cup(c_1,b]=[a,b].
    $$
    2,3. Since $U_k\subset K(N_{k-1})\subset \Jall(N_{k-1})$, $U_k\cap N_i=\emptyset$ for $i\leq k-1$.
    Since $U_{k+1}\subset K(N_{k})\subset \Jall(N_{k})$, $U_{k+1}\cap N_k=\emptyset$. 
    More generally, since $K(N_{k+1})\subset K(N_k)$, $U_{i}\cap N_k=\emptyset$ for $i\geq k+1$. Since every point in $[a,b]$ belongs to some $U_i$, the only possibility if that $N_k\subset U_k$ for each $k=0,\dots,p-1$.
\end{proof}
\allblack
\begin{lemma}
    Let $f$ be a T-unimodal map. Then $x\in[a,b]$ has a backward orbit asymptoting to $a$ if and only if $x\in[a,c_1]$.
\end{lemma}
\begin{proof}
    Consider the interval $K_\varepsilon=[a,\varepsilon]$ for any $\varepsilon>0$ small enough. Since $f$ has no homtervals, there is some integer $k>0$ such that $c\in f^k(K_\varepsilon)$ and so $f^{k+1}(K_\varepsilon)=[a,c_1]$. This means that every point in $[a,c_1]$ has a backward orbit segment getting arbitrarily close to $a$, and so a backward orbit asymptoting to $a$.
\end{proof}
\begin{lemma}
    Let $0\leq k\leq p-1$. Then $K(N_k)$ is the set of all points of $\overline{\Jall(N_{k})}$
    that have a backward trajectory asymptoting to $N_{k+1}$.
\end{lemma}
\begin{proof}
  Let $r$ be the period of $\cT(N_k)$ and set $\bar f = f^{r}$. 
  Then each $J_i(N_k)$ is forward-invariant under $\bar f$. 
  After restricting $\bar f$ to each of the $J_i(N_k)$, we reduce to the following problem: given a T-unimodal map $\bar f$ with domain $[\bar a,\bar b]$ and at least two repelling nodes $\overline{N}_0=\{\bar a\}$ and $\overline{N}_1$,
  determine which points have a backward orbit asymptoting to $\overline{N}_1$.
  
  Recall that $\overline{N}_1\subset K(\overline{N}_0)$. 
  Since $K(\overline{N}_0)=[\bar c_2,\bar c_1]$ is a trapping region, there cannot be any  backward orbit $\{\dots,x_{-2},x_{-1},x_0\}$ such that $x_{-i}\not\in K(\overline{N}_0)$ and $x_{-i-1}\in K(\overline{N}_0)$, since $f(x_{-i-1})=x_{-i}$ and so there would be a forward orbit getting out of $K(\overline{N}_0)$. 
  Hence, only points in $K(\overline{N}_0)$ can asymptote backward to $\overline{N}_1$. 
  
  Assume first that $\overline{N}_1$ is a fixed point and
  consider the cyclic trapping region $\cT(\overline{N}_1)=\{\bar J_1=[\bar q_1,\bar p_1],\bar J_2=[\bar p_1,\bar q_2]\}$. 
  By the Lemma above, applied to the restrictions of $\bar f^2$ to $\bar J_1$, a point $x\in \bar J_1$ has an orbit backward asymptotic to $\bar p_1$ if and only if $x\in[\bar f^2(c),\bar p_1]=[\bar c_2,\bar p_1]$. 
  In turn, since $\bar f:\bar J_2\to\bar J_1$ is a homeomorphism, this means that $x\in\bar J_2$ has an orbit backward asymptotic to $\bar p_1$ if and only if $x\in [\bar p_1,\bar c_1]$. 
  Ultimately, this shows that every $x\in K(\overline{N}_0)$ 
  has an orbit backward asymptotic to $\overline{N}_1$.

  Assume now that $\overline{N}_1$ is a Cantor set. 
  Then a point $x\in K(\overline{N}_0)$ is either in $\overline{N}_1$ or there is some integer $r$ such that $x\in\bar f^{-r}(J_1(\overline{N}_1))$.
  In the first case, the action of $\bar f$ on $\overline{N}_1$ is a subshift of finite type and each point of a subshift of finite type has a backward dense orbit (see Thm.~2 in~\cite{DL22}). 
  In the second case, recall that, by construction, $\overline{N}_1$ is the complement in $K(\overline{N}_0)$ of all counterimages of $J_1(\overline{N}_1)$.
  Hence, for any integer $r>0$, every point belonging to $\bar f^{-r}(J_1(\overline{N}_1))\cap K(\overline{N}_0)$ has orbits backward asymptotic to $\overline{N}_1$.

\end{proof}
%


\begin{definition}[\cite{Her92}]
  Given a discrete dynamical system $f$ on $X$, the {\em special $\alpha$-limit} $s\alpha_f(x)$ of a point $x\in X$ is the union of all limit points of all backward trajectories under $f$ based at $x$.
\end{definition}

\allblack
\begin{theorem}[{\bf $\bm{s\alpha}$-limits of T-unimodal maps}]
  \label{thm:main}
  Let $f$ be a T-unimodal map with attractor $A$ and $p+1$ nodes $N_0,\dots,N_p$. 
  If $x$ is a level-$k$ point, $k<p$, then $$s\alpha_f(x)=\bigcup\limits_{i=0}^kN_i.$$
  \teal{If $N_p$ is not of type $A_5$ and $x$ is a level-$p$ point, then}
  $$s\alpha_f(x)=\bigcup\limits_{i=0}^pN_i=\Omega_f=\Rf.$$ 
  If $N_p$ is of type $A_5$ and $C$ is the repelling Cantor set in $N_p$, then:
  \begin{enumerate}
      \item $s\alpha_f(x)=\bigcup\limits_{i=0}^{p-1}N_i\cup C\cup A=\Omega_f\subsetneq\Rf$ for every level-$p$ point $x$ in $A$;
      \item $s\alpha_f(x)=\bigcup\limits_{i=0}^{p-1}N_i\cup C$ for every other level-$p$ point $x$.
  \end{enumerate}
\end{theorem}
\begin{proof}
\allmagenta
  Let $x$ be a level-$k$ point for $k<p$. 
  Then $x$ belongs to $K(N_i)$ for $i=0,\dots,k-1$ and so, by the Lemma above, it has backward orbits to each $N_j$ for $j=0,\dots,k$.
  This shows that $s\alpha_f(x)=\cup_{i=0}^kN_i$ for every level-$k$ point, $k<p$.

  Consider now the case $k=p$. As showed in Thm.~\ref{thm:dense}, each point of the attractor $A\subset N_p$ has a backward trajectory dense in $A$.
  Hence, if $N_p$ is of type $A_2$,
  $s\alpha_f(x)=\cup_{i=0}^pN_i=\Omega_f$ for every level-$p$ point.

  When $N_p$ is of type $A_5$, the node contains $A$, a repelling Cantor set $C$ and part of the basin. 
  The points in the basin are wandering points and these points cannot be obtained as limit of backward trajectories.
  Recall that, in this case, the attractor is a {\em cyclic} trapping region $\{J_1,\dots,J_r\}$ and the repelling $C$ is the set obtained from $N_p$ after removing from it all counterimages
  of $J_1$.   
  Hence, if $x\in A$ then its $s\alpha$-limit contains both $A$, because each
  point in a chaotic attractor has a backward dense orbit, and $C$, because of the relation
  mentioned above between $A$ and $C$, so that $s\alpha_f(x)=\Omega_f$.
  For the same reasons, if $x\in N_p\setminus A$, then $s\alpha_f(x)=\cup_{i=0}^{p-1}N_i\cup C$.
\end{proof}
The sets $V_i$ in the claim of Thm.~\ref{thm:A} are related to the $U_j$ by the following relation: $V_i = \cup_{j=i}^p U_j$.
\section*{Acknowlegments}
This work was supported by the National Science Foundation, Grant No. DMS-1832126.
\bibliographystyle{amsplain}  
\bibliography{refs}  

\end{document}